\documentclass[11pt]{article}
\usepackage[utf8]{inputenc}
\usepackage{amssymb,amsmath,amsthm,amsfonts, graphicx,geometry,hyperref,xcolor,bm,mathtools,authblk,subcaption,multirow,tensor,mathabx,tikz,algorithm,algpseudocode, pifont, booktabs, subcaption}
\usepackage[numbers]{natbib}
\usepackage{authblk}

\newgeometry{left=1in,right=1in,top=1in,bottom=1in}

\hypersetup{
    colorlinks=true,
    linkcolor=blue,
    filecolor = blue
, urlcolor = blue,
citecolor = blue
}

\newtheorem{theorem}{Theorem}

\newtheorem{corollary}{Corollary}
\newtheorem{lemma}{Lemma}

\theoremstyle{definition}

\newtheorem{example}{Example}

\begin{document}

\title{Proper Bayes minimax multiple shrinkage estimation}
\author{ Pankaj Bhagwat\thanks{Corresponding author: \texttt{pbhagwat@ualberta.ca}}\\ Department of Mathematical and Statistical Sciences\\ University of Alberta \\ \texttt{pbhagwat@ualberta.ca} \and William E. Strawderman\\ Department of Statistics\\ Rutgers University\\ \texttt{straw@stat.rutgers.edu} \and Edward I. George\\ Department of Statistics and Data Science\\ The Wharton School, University of Pennsylvania\\ \texttt{edgeorge@wharton.upenn.edu} } 
\date{\today}

\maketitle

\begin{abstract}
For the canonical problem of estimating a multivariate normal mean under squared error loss, we demonstrate, for the first time, the existence of proper Bayes minimax multiple shrinkage estimators by introducing a general approach for their explicit construction.  As opposed to  minimax shrinkage estimators that shrink towards a single prespecified target, minimax multiple shrinkage estimators adaptively shrink towards the more promising of a set of prespecified targets, substantially increasing the region of potential risk reduction while maintaining the protection of always being at least as good as the maximum likelihood estimator. These estimators are particularly useful in practice as they address the challenge of selecting a minimax shrinkage estimator when prior information suggests more than one viable shrinkage target to choose from.  In contrast to previous formal Bayes minimax multiple shrinkage estimators, which were built on mixtures of superharmonic marginals, these proper Bayes minimax multiple shrinkage estimators are obtained via mixtures of square-root superharmonic marginals.   Examples of such proper Bayes minimax multiple shrinkage estimators include an adaptive convex combination of the rescaled Strawderman shrinkage estimators.
\end{abstract}
   

\section{Introduction}

Consider the canonical problem of estimating the mean of a multivariate normal distribution with identity covariance under squared error loss.  More precisely, based on observing $X \sim N_d(\theta,I_d)$,  a normal random vector in $\mathbf{R}^d$  with mean  $\theta$ and  covariance matrix $I_d$, the problem is to estimate $\theta \in \mathbf{R}^d$ by an estimator $\delta(X)$ which yields small quadratic risk 
\begin{equation}   \label{Quadraticloss}  
 R(\theta, \delta) = \mathbf{E}_{\theta}||\delta(X) - \theta||^2.
\end{equation}

As is well known, the maximum likelihood estimator (mle) for this problem, $\delta_{mle}(X) = X$, is best invariant and minimax with constant risk $ R(\theta, \delta_{mle}) \equiv d$. Although admissible in dimensions $d=1$ and $d=2$  (see \cite{hodges1951some, james1961proceedings}),   Stein in \cite{stein1956} demonstrated the striking  result that the mle is no longer admissible  in three or more dimensions, and followed that up in \cite{james1961proceedings} by introducing the James-Stein estimator, an explicit minimax shrinkage estimator that not only dominated the mle, but by shrinking $X$ towards the target 0, offered substantial risk reduction in a neighborhood of $0$.   Note that as a practical matter, if prior information suggested that $\theta$ was likely to be close to some $\theta_0 \ne 0$, it would be more sensible to use a minimax translated version of the James-Stein estimator that would shrink towards $\theta_0$.  Stein's work set in motion a hunt for improved minimax shrinkage estimators. 
Continuing to this day, it has led to the development of a wide variety of new shrinkage estimators and techniques which improve upon the mle and achieve minimaxity in higher dimensions \cite{FSW2018}. 

A limitation of a single target minimax shrinkage estimator such as the James-Stein estimator is that its meaningful risk reduction is confined to a relatively small region of the parameter space surrounding the target.  But what happens if a priori knowledge suggests  more than one reasonable target region to shrink towards? This dilemma of choosing  minimax shrinkage estimators when there are multiple viable targets available was addressed with the introduction of minimax multiple shrinkage estimators \cite{george1986a}, \cite{george1986c} and \cite{george1986b}. These estimators adaptively shrink $X$ toward the more promising of a set of targets according to the observed data, thereby increasing the region of potential risk reduction while maintaining the protection of always offering smaller risk than the maximum likelihood estimator.  Motivated as posterior mean Bayes estimators under a mixture of priors centered at each of the prespecified targets, the minimaxity of these multiple shrinkage estimators was established under the condition that the component marginals be superharmonic.  Because such marginals must necessarily be improper \cite{FSW1998}, this condition ruled out the possibility of proper Bayes  minimax multiple shrinkage estimators.  Indeed, up to now, even just the existence of proper Bayes minimax multiple shrinkage estimation has remained an open question.  

In this paper, we are able to finally answer this question by showing that multiple shrinkage estimators induced by mixtures of proper priors with square-root superharmonic marginals can in fact achieve minimaxity.  Desirable from both the theoretical and practical point of view, these proper Bayes minimax multiple shrinkage estimators are automatically admissible. Additionally, the weights associated with the underlying proper mixture priors can be coherently construed as the prior probabilities that each of the component priors serves as the generating mechanism for the target values.  With simulations of a particular minimax multiple shrinkage estimator, namely an adaptive convex combination of Strawderman estimators, the risk reduction at each of the targets is seen to be nearly the same as the risk reduction of the best of the component single target Strawderman estimators.  

In Section \ref{Shrinkage estimation section}, we begin by considering the single target shrinkage estimator representations proposed by Brown \cite{Brown1971} and Baranchik \cite{Baranchik1970}, and show how Stein's superharmonic and square-root superharmonic minimaxity conditions carry over from the former to the latter.  In Section \ref{Multiple shrinkage pseudo-Bayes estimator}, we develop minimaxity conditions for  multiple shrinkage estimators in terms of both the Brown and Baranchik representations of their component estimators.  We do this for both the known superharmonic marginal component conditions, and for our new square-root superharmonic marginal conditions (Theorem \ref{MainResult}), which are fundamental for the validity of our proper Bayes minimax construction.  In Section \ref{Scaled marginals section}, we propose a rescaling modification of marginal density functions to satisfy a sometimes needed flattening condition for the square-root superharmonic case. In Subsection \ref{An example of a pseudo-Bayes multiple shrinkage estimator which is minimax}, we illustrate the use of such rescaling in the construction of a simple class of pseudo Bayes multiple shrinkage estimators.  In Section \ref{Adjusting priors yielding scaled marginals}, we present a class of priors which lead to marginal densities with the scaling property introduced in Section \ref{Scaled marginals section} and show that the associated multiple shrinkage estimators are minimax. We conclude this section by providing examples of proper Bayes multiple shrinkage estimators which are minimax for $d \geq 5$. Finally, in Section \ref{Simulation Study}, we illustrate via simulations the risk reduction potential of proper Bayes multiple shrinkage estimators induced by mixtures of Strawderman priors.

\section{Minimax Bayes shrinkage estimation}
\label{Shrinkage estimation section}
Given an observation $x$ of $X \sim N_d(\theta,I_d)$, Brown \cite{Brown1971} showed that the posterior mean Bayes estimator $\delta_{\pi}$ of $\theta$ under a prior $\pi$ takes the form
\begin{align} \label{Brown1}
   \delta_{\pi}(x) = x + \frac{\nabla m_{\pi}(x)}{m_{\pi}(x)},\; \hspace{8mm} 
\end{align}
where $m_{\pi}(x)$ is the marginal density of $X$ under $\pi$ and $\nabla = (\partial/\partial x_1,\ldots, \partial/\partial x_d)^{\top}$, is the gradient operator.  Expanding the risk of $\delta_{\pi}$  around this representation, coupled with an integration by parts, Stein \cite{s1981} obtained the following unbiased estimate of  the quadratic risk difference $R(\theta,\delta_\pi) - R(\theta, \delta_{mle})$, 
\begin{align} \label {UBER1}
   \hat{\Delta}\delta_\pi(x) & = \frac{2\nabla^2 m_{\pi}(x)}{m_{\pi}(x)} - \frac{||\nabla m_{\pi}(x)||^2}{m_{\pi}^2(x)} \\
  \label {UBER2} & = 4\nabla^2 \sqrt{m_\pi(x)}/\sqrt{m_\pi(x)},
\end{align}
where $\nabla^2 = \sum\limits_{j = 1}^{d}\partial^2/\partial x_j^2$ is the Laplacian.  Note that $E_\theta   \hat{\Delta}\delta_\pi(X)  = R(\theta,\delta_\pi) - R(\theta, \delta_{mle})$  for all $\theta$ whenever  the expectation of the rightmost term in (\ref{UBER1}) is finite.  

Because $\hat{\Delta}\delta_\pi(x)$ does not depend on $\theta$, it follows from (\ref{UBER1}) that $\delta_\pi$ dominates $\delta_{mle}$ and is thereby minimax whenever $m_{\pi}$ is superharmonic, namely $\nabla^2 m_{\pi}(x) \leq 0$ for all $x \in \mathbf{R}^d$.  The weaker condition of square-root superharmonicity for $m_{\pi}$, namely $\nabla^2 \sqrt{m_{\pi}(x)} \leq 0$  for all $x \in \mathbf{R}^d$, is seen to be  sufficient for minimaxity from (\ref{UBER2}).

Of particular interest for us will be Bayes estimators of $\theta$ under spherically symmetric prior densities of the form 
 \begin{equation}
    \nonumber \pi(||\theta||^2).
 \end{equation}
Letting $t = ||x||^2$,  such priors lead to marginals of the form $m_{\pi}(x) = f_{\pi}(t)$, whereby  $\delta_{\pi}$ in (\ref{Brown1})  can be expressed as
 \begin{align} \label{Brown2}
   \delta_{\pi}(x) = x + \frac{2f_{\pi}'(t)}{f_{\pi}(t)} x. \hspace{8mm} 
\end{align}

A further representation of $\delta_{\pi}$ proposed by Baranchik \cite{Baranchik1970}, which will be key in our development, is
\begin{align}
   \delta_{\pi}(x) = x - \frac{r(t)}{t}\;x\; 
   \label{rform}
\end{align}
where
\begin{align}
    \frac{r(t)}{t} = - \frac{2f_{\pi}'(t)}{f_{\pi}(t)}.
    \label{rFunctiondef}
\end{align}
Note that when $r(t) \equiv (d-2)$,  (\ref{rform}) reduces to the James-Stein estimator, and also that $r(t) \geq 0$ when $f_{\pi}'(t) \leq 0.$

The general form (\ref{rform}), without appealing to any connection with a Bayes estimator, was used by \cite{Baranchik1970} and \cite{s1971} to establish minimaxity conditions for $\delta_{\pi}$ based only on monotonicity assumptions with suitable bounds on $r(\cdot)$.  In the following, we obtain those same conditions for $\hat{\Delta}\delta_\pi(x) \leq 0$  by connecting $r(t)$ to $f_{\pi}(t)$ via (\ref{rFunctiondef}).

  \begin{lemma}  \label{Lemma2.1}
 If $r'(t) \ge 0$ and $0 \le r(t) \leq 2(d-2)$ for $t = ||x||^2$, then $\nabla^2 \sqrt{m_{\pi}(x)} \leq 0$ and $\hat{\Delta}\delta_\pi(x) \leq 0$.
  \end{lemma}

  \begin{proof} 
In terms of these representations, Stein's unbiased estimate of the quadratic risk difference $R(\theta,\delta_\pi) - R(\theta, \delta_{mle})$ becomes
 \begin{align}
    \nonumber  \hat{\Delta}\delta_\pi(x) & = 2\;\frac{2 d f_{\pi}'(t) + 4 t f_{\pi}''(t)}{f_{\pi}(t)} - 4\left(\frac{f_{\pi}'(t)}{f_{\pi}(t)}\right)^2t \\
     \nonumber & = \frac{r^2(t)}{t} - \frac{2(d-2)r(t)}{t} - 4r'(t) \\
& \leq \frac{r^2(t)}{t} - \frac{2(d-2)r(t)}{t},
     \label{SUREexpression}
 \end{align}
where the last inequality follows from $r'(t) \geq 0$, which we shall be assuming throughout. Thus $\hat{\Delta}\delta_\pi(x) \leq 0$ follows from  (\ref{SUREexpression}), and then $\nabla^2 \sqrt{m_{\pi}(x)} \leq 0$ follows from (\ref{UBER2}).
 \end{proof}

 \begin{lemma}  \label{Lemma2.2}
If $r'(t) \ge 0$ and $0 \le r(t) \leq (d-2)$ for $t = ||x||^2$, then $\nabla^2  m_{\pi}(x) \leq 0$ and $\hat{\Delta}\delta_\pi(x) \leq 0$.
 \end{lemma}
 
 \begin{proof}
     From $f_{\pi}'(t) = -\frac{1}{2}\frac{r(t) f_{\pi}(t)}{t}$, we have 
     \begin{align}
        \nonumber  f_{\pi}''(t) &  = -\frac{1}{2}\left[-\frac{r(t) f_{\pi}(t)}{t^2} + \frac{r'(t) f_{\pi}(t)}{t} + \frac{r(t) f_{\pi}'(t)}{t}\right] \\
      \nonumber  & =  -\frac{1}{2}\left[-\frac{r(t)f_{\pi}(t)}{t^2} + \frac{r'(t) f_{\pi}(t)}{t}  - \frac{1}{2}\frac{r^2(t) f_{\pi}(t)}{t^2}\right].
     \end{align}
Hence, 
\begin{align}
    \nonumber \nabla^2 m_{\pi}(x) & =  2 d f_{\pi}'(t) + 4 t f_{\pi}''(t) \\
    \nonumber & = -d\frac{r(t) f_{\pi}(t)}{t} - 2t\left[-\frac{r(t) f_{\pi}(t)}{t^2} + \frac{r'(t) f_{\pi}(t)}{t}  - \frac{1}{2}\frac{r^2(t) f_{\pi}(t)}{t^2}\right] \\
    \nonumber & = f_{\pi}(t) \left[ -(d - 2)\frac{r(t)}{t} + \frac{r^2(t)}{t} - 2r'(t) \right] \\
    \nonumber & \leq f_{\pi}(t) \frac{r(t)}{t}\left[ -(d - 2) + r(t) \right]. 
\end{align}
Thus, $\nabla^2 m_{\pi}(x) \leq 0$ when $r(t) \leq (d - 2)$, and then
$\hat{\Delta}\delta_\pi(x) \leq 0$ follows from (\ref{UBER1}). 
 \end{proof}

Note that when $r'(t) \ge 0$ and $0 \leq r(t) \leq (d-2)$ for all $t$, $\delta_\pi$ will be minimax and driven by a marginal $m_{\pi}$ that is superharmonic. As shown by \cite{FSW1998}, a superharmonic $m_{\pi}$ cannot be integrable and so cannot be induced by a proper prior. However, when $r'(t) \ge 0$ and $0 \leq r(t) \leq b \equiv \sup r(\cdot)$ where $(d-2) < b \le 2(d-2)$ for all $t$, then $\delta_\pi$ will be minimax, but driven by a marginal that is square-root superharmonic but not superharmonic. Such a marginal $m_{\pi}$ and its corresponding $\delta_\pi$ can in fact be induced by a proper prior when $d \geq 5$.

\section{Multiple shrinkage estimation}
\label{Multiple shrinkage pseudo-Bayes estimator} 
    Let us now consider the construction of the minimax multiple shrinkage estimators introduced in \cite{george1986a}, \cite{george1986c} and \cite{george1986b}.  To begin with, note that minimax Bayes estimators under spherically symmetric priors of the form $\pi(||\theta||^2)$,  shrink $X$ towards $\theta = 0$ where their smallest risk is obtained.  However, if prior information actually suggested that $\theta$ was more likely to be close to some $\theta_0 \ne 0$, it would be more sensible to deploy a re-centered spherically symmetric prior of the form $\pi(||\theta - \theta_0||^2)$, as this would generate a minimax estimator shrinking towards $\theta_0$ where its smallest risk would be obtained.  

Now suppose prior information suggested only that $\theta$ was likely to be close to one of a number of different values $\theta_1,\ldots,\theta_k$.  From a Bayesian point of view, this information would be naturally captured by a mixture of the re-centered priors $\pi(||\theta - \theta_i||^2)$ at each of these values, namely
\begin{align}\label{mixtureprior}
    \pi_{*}(\theta) = \sum\limits_{i=1}^{k}\;w_i\;\pi(||\theta - \theta_i||^2),
\end{align}
where $w_1, \ldots,w_k$ are prespecified fixed positive weights that satisfy $\sum_{i = 1}^k w_i = 1$.
Note that these weights can be coherently interpreted as valid prior probabilities only when the generating prior form $\pi$ is proper (integrable).  

Under $\pi_{*}$, the marginal density of $X$ is a mixture of the component marginals,
\begin{align} \label{mstar}
m_*(x) = \sum\limits_{i=1}^{k}\;w_i\;m_i(x),
\end{align}
where each $m_i(x) = f_{\pi}(||x - \theta_i||^2)$ is the marginal density of $X$ under $\pi(||\theta - \theta_i||^2)$.
This marginal $m_*(x)$ then leads immediately, via Brown's representation (\ref{Brown1}), to a multiple shrinkage estimate of the form
\begin{align}
\delta_{*}(x) = x + \frac{\nabla m_*(x)}{m_*(x)} = \sum_{i=1}^k \rho_i(x) \delta_\pi^{\theta_i} (x),
\label{MultipleShrinkageEstimator}
\end{align}
an adaptive convex combination of the  shrinkage estimates
\begin{align}\label{mstarcomponents}
\delta_\pi^{\theta_i}(x) = x + \frac{\nabla m_i(x)}{m_i(x)} =  x - \frac{ r(t_i)}{t_i}(x-\theta_i)
\end{align}
where $ t_i = \|x-\theta_i\|^2$ for $ i = 1,\ldots,k$.  Note that the adaptive weights 
\begin{align}
\rho_i(x)  = \frac{w_i m_i (x)}{ m_*(x)},
\end{align}
put more weight on those $\delta_\pi^{\theta_i}(x)$ that are closest to their targets and shrinking most.
In the same vein as the prior weights $w_1, \ldots,w_k$, these adaptive weights $\rho_1(x),\ldots \rho_k(x)$ can be coherently interpreted as posterior probabilities only when the generating prior $\pi$ is proper (integrable).

As observed by \cite{george1986a},  $m_*$ will be superharmonic, and hence $\delta_{*}$ minimax, when its component marginals $m_i$ are all superharmonic, a simple consequence of the linear relationship (\ref{mstar}).  In terms of the $r(\cdot)$ function driving the component estimators $\delta_\pi^{\theta_i}$ in (\ref{mstarcomponents}), these results are captured by the following Lemma and Theorem.  

\begin{lemma} \label{Lemma3.1}
For $x \in \mathbf{R}^d$ let $t_i = ||x - \theta_i||^2$, $i = 1,\ldots,k$.   If $r'(t_i) \geq 0$ and $0 \leq r(t_i) \leq (d - 2)$ for every $\delta_\pi^{\theta_i}$ in (\ref{mstarcomponents}), then $\nabla^2 m_*(x) \leq 0\;.$
\end{lemma}
\begin{proof}
By Lemma \ref{Lemma2.2}, $\nabla^2 m_i (x) \leq 0$ for $i = 1,\ldots,k$.  It then follows from (\ref{mstar}) that $\nabla^2 m_*(x) = \sum_{i = 1}^k \;w_i\,\nabla^2 m_i(x) \leq 0$.
\end{proof}

The following is now immediate.
\begin{theorem}
    If $r'(t) \ge 0$ and $0 \leq r(t) \leq (d-2)$ for all $t$, then $m_*$ is superharmonic and $\delta_*$ in (\ref{MultipleShrinkageEstimator}) is minimax.
\end{theorem}

When $0 \leq r(\cdot) \leq (d-2)$, $\delta_*$ in (\ref{MultipleShrinkageEstimator}) achieves minimaxity as an adaptive convex combination of generalized Bayes minimax shrinkage estimators driven by superharmonic marginals, which cannot be induced by proper priors.  In this case,  $\delta_*$ is also limited to be a generalized Bayes estimator as its underlying mixture prior $\pi_{*}$ in (\ref{mixtureprior}) must necessarily be improper.  

A natural question to ask is whether it is possible for $\delta_*$ to achieve minimaxity as an adaptive convex combination of proper Bayes minimax shrinkage estimators, each driven by square-root superharmonic marginals.  The construction of such a $\delta_*$ in (\ref{MultipleShrinkageEstimator}) would be obtained when $0 \leq r(\cdot) \leq b = \sup r(\cdot)$ with $(d - 2) < b < 2(d - 2)$. However, establishing the minimaxity of such $\delta_*$ is far less straightforward than it was in the superharmonic component case.  To see why, note that while $\nabla^2 m_*(x) =\sum_1^k\;w_i\,\nabla^2 m_i(x)$, implied by (\ref{mstar}), made it easy to establish the superharmonicity of $m_*$ in that case, no such relationship exists between $\nabla^2\sqrt{m_*(x)}$ and the $\nabla^2\sqrt{m_i(x)}$ in the square-root superharmonic case. Indeed, up to now, even the mere existence of a proper Bayes minimax multiple shrinkage estimator has been an open question.  In the following, we are pleased to be able to finally answer this question in the affirmative and demonstrate the construction of specific cases of such $\delta_*$ which are ready for practical implementation.  \\ 

For the multiple shrinkage estimator $\delta_*$ defined by (\ref{MultipleShrinkageEstimator}) let  $D = \max\limits_{i\neq j} ||\theta_i - \theta_j||^2$ denote the maximum squared distance possible between any two targets in  $\theta_1,\ldots,\theta_k$ .  For $\rho \in (0,1)$, let 
\begin{align}
    \textbf{A} = \left\{x \in \mathbf{R}^d: \min\limits_{1\leq i \leq k} ||x - \theta_i||^2 = \min\limits_{1 \leq i \leq k} t_i \geq  t_0 = \frac{D}{\rho^2}\right\}.
  \label{A1Condition}
\end{align}

\begin{lemma}
    Suppose $r'(t) \ge 0$, $0 \leq r(t) \leq b = \sup r(\cdot)$ with $(d - 2) < b < 2(d - 2)$,  and $r(t)/t$ is non-increasing for $t > 0$. Then  $\hat{\Delta}\delta_*(x) \le 0$ for all $x \in \textbf{A}$ when $\rho \in \left(0,\;\frac{2(d - 2) - b}{2b - 2(d - 2)}\right]$.
    \label{LemmaA1brho}
\end{lemma}
\begin{proof}
For any $x$, the unbiased risk difference estimate $\hat{\Delta}\delta_*(x)$ can be expressed as
\begin{align}
  \hat{\Delta}\delta_*(x) = \frac{2\;\sum\limits_{i=1}^{k}\;w_i\;\nabla^2 m_i(x)}{\sum\limits_{i=1}^{k}\;w_i\;m_i(x)} - \frac{||\sum\limits_{i=1}^{k}\;w_i\;\nabla m_i(x)||^2}{\left(\sum\limits_{i=1}^{k}\;w_i\;m_i(x)\right)^2}\; ,
  \label{Ddelta*}
\end{align}
so that
\begin{align}
   \nonumber  \left(\sum\limits_{i=1}^{k}\;w_i\;m_i(x)\right)^2 \hat{\Delta}\delta_*(x)  & = 2\;\left(\sum\limits_{i=1}^{k}\;w_i\;\nabla^2 m_i(x) \right)\left(\sum\limits_{i=1}^{k}\;w_i\;m_i(x)\right) \\
  \nonumber  & \quad \hspace{5mm} - ||\sum\limits_{i=1}^{k}\;w_i\;\nabla m_i(x)||^2 \\
    \nonumber & = \sum\limits_{i=1}^{k}\;w_i^2\;\left(2\nabla^2 m_i(x) m_i(x) - ||\nabla m_i(x)||^2\right) \\
  &  \hspace{9mm} + \sum\limits_{i\neq j} w_i w_j m_i(x) m_j(x) \left[\frac{2\nabla^2 m_i(x)}{m_i(x)} - \frac{\nabla m_i(x)^{\top}\nabla m_j(x)}{m_i(x) m_j(x)}\right].
   \label{eqIandII}
\end{align}

Thus $\hat{\Delta}\delta_*(x) \le 0$  for $x \in \textbf{A}$, will follow if both\\
\begin{align} 
2\nabla^2 m_i(x) m_i(x) - ||\nabla m_i(x)||^2 \leq 0 \quad \forall i= 1,\ldots,k,
\label{PartI}
\end{align}
and
\begin{align} 
 \frac{2\nabla^2 m_i(x)}{m_i(x)} - \frac{\nabla m_i(x)^{\top}\nabla m_j(x)}{m_i(x) m_j(x)} \leq 0 \; , \quad \forall \;i\neq j.
 \label{PartII}
\end{align}

To begin with, (\ref{PartI}) follows immediately from Lemma \ref{Lemma2.1} for any $x$, since   $(2\nabla^2 m_i(x) m_i(x) - ||\nabla m_i(x)||^2)$ = $m_i^2(x) \hat{\Delta}\delta_\pi^{\theta_i}(x)$ .

To show (\ref{PartII}) for $x \in \textbf{A}$, let $t_i  = \|x - \theta_i\|^2$,  $t_j  = \|x - \theta_j\|^2$, and  $D_{ij}  = \|\theta_i - \theta_j\|^2$ so that   for any $i \neq j$,
   \begin{align}
   \nonumber  \frac{2\nabla^2 m_i(x)}{m_i(x)} - \frac{\nabla m_i(x)^{\top}\nabla m_j(x)}{m_i(x) m_j(x)} & =  \frac{2r(t_i)}{t_i}\left\{-(d-2) + r(t_i)\right\} - 4r'(t_i) - \frac{r(t_i)}{t_i}\;\frac{r(t_j)}{t_j}(x - \theta_i)^{\top}(x - \theta_j) \\
   \nonumber & \leq \frac{r(t_i)}{t_i} \left\{-2(d-2) + 2r(t_i) - \frac{r(t_j)}{t_j}(x - \theta_i)^{\top}(x - \theta_j)\right\} \\
    \nonumber & = \frac{r(t_i)}{t_i} \left\{-2(d-2) + r(t_i)\left[ 2  - \frac{r(t_j)}{r(t_i)} \frac{(x - \theta_i)^{\top}(x - \theta_j)}{t_j}\right]\right\} \\
    & = \frac{r(t_i)}{t_i} \left\{-2(d-2) + r(t_i)\left[ 2  - \frac{r(t_j)}{r(t_i)} \frac{t_i + t_j - D_{ij}}{2t_j}\right]\right\}
    \label{Riskestimatorforcross}
\end{align}
where the final equality follows from $(x - \theta_i)^{\top}(x - \theta_j) = \frac{t_i + t_j - D_{ij}}{2}$, a consequence of
$\|(x - \theta_i) - (x - \theta_j)\|^2 = \|(x - \theta_i)\|^2 + \| (x - \theta_j)\|^2 - 2(x - \theta_i)^{\top}(x - \theta_j)$.

To show that the right-hand side of (\ref{Riskestimatorforcross}) is non-positive, we consider two cases:\\
  
  {Case (i) \hspace{-1mm}:} For $t_j \geq t_i$, since $r(t)$ is non-decreasing,
  \begin{align}
    \nonumber   \frac{r(t_j)}{r(t_i)}\; \frac{t_i + t_j - D_{ij}}{2t_j} \geq \frac{t_i + t_j - D_{ij}}{2t_j} = \frac{1}{2} + \frac{t_i  - D_{ij}}{2t_j}.
  \end{align}
  
  {Case (ii) \hspace{-1mm}:} For $t_j < t_i$, since $r(t)/t$ is non-increasing,
  \begin{align}
   \nonumber   \frac{r(t_j)}{r(t_i)}\; \frac{t_i + t_j - D_{ij}}{2t_j} =  \frac{r(t_j)/t_j}{r(t_i)/t_i} \; \frac{t_i + t_j - D_{ij}}{2t_i}  \geq \frac{t_i + t_j - D_{ij}}{2t_i} = \frac{1}{2} + \frac{t_j  - D_{ij}}{2t_i}.
  \end{align}
\\
From the triangle inequality, we have 
\begin{align}
    \label{triangle inequalities1}  |\sqrt{t_i} - \sqrt{D_{ij}}| \leq \sqrt{t_j} \leq  \sqrt{t_i} + \sqrt{D_{ij}} \\
    |\sqrt{t_j} - \sqrt{D_{ij}}| \leq \sqrt{t_i} \leq  \sqrt{t_j} + \sqrt{D_{ij}} 
    \label{triangle inequalities2}
\end{align}
from which we obtain
\begin{align*}
    \frac{t_i  - D_{ij}}{2t_j} \geq \frac{t_i  - D_{ij}}{2(\sqrt{t_i} + \sqrt{D_{ij}})^2}  & = \frac{(\sqrt{t_i} + \sqrt{D_{ij}})(\sqrt{t_i} - \sqrt{D_{ij}})}{2(\sqrt{t_i} + \sqrt{D_{ij}})^2}  = \frac{(\sqrt{t_i} - \sqrt{D_{ij}})}{2(\sqrt{t_i} + \sqrt{D_{ij}})}\;\\
    & = \frac{1}{2} - \frac{\sqrt{D_{ij}}}{(\sqrt{t_i} + \sqrt{D_{ij}})} \geq \frac{1}{2} - \frac{\sqrt{D_{ij}}}{(\sqrt{t_0} + \sqrt{D_{ij}})} \\
    & = \frac{1}{2} - \frac{\sqrt{D_{ij}}}{(\sqrt{D}/\rho + \sqrt{D_{ij}})} = \frac{1}{2} - \frac{\rho }{(\sqrt{D}/\sqrt{D_{ij}} + \rho)} \\
    & \geq \frac{1}{2} - \frac{\rho }{(1 + \rho)},
\end{align*}
and similarly,
\begin{align*}
    \frac{t_j  - D_{ij}}{2t_i} \geq \frac{t_j  - D_{ij}}{2(\sqrt{t_j} + \sqrt{D_{ij}})^2}  & \geq \frac{1}{2} - \frac{\rho }{(1 + \rho)}.
\end{align*}

It then follows that\\

{Case (i) \hspace{-1mm}:} For $t_j \geq t_i$, 
  \begin{align}
    \nonumber   \frac{r(t_j)}{r(t_i)}\; \frac{t_i + t_j - D_{ij}}{2t_j} \geq  \frac{1}{2} + \frac{t_i  - D_{ij}}{2t_j} \geq 1 - \frac{\rho }{(1 + \rho)} = \frac{1}{1 + \rho}.
  \end{align}
  \\
  
{Case (ii) \hspace{-1mm}:} For $t_j < t_i$, 
  \begin{align}
   \nonumber   \frac{r(t_j)}{r(t_i)}\; \frac{t_i + t_j - D_{ij}}{2t_j} =  \frac{r(t_j)/t_j}{r(t_i)/t_i} \; \frac{t_i + t_j - D_{ij}}{2t_i}  \geq \frac{1}{2} + \frac{t_j  - D_{ij}}{2t_i}\geq 1 - \frac{\rho }{(1 + \rho)} = \frac{1}{1 + \rho}.
  \end{align}
\\

Hence, the right-hand side of (\ref{Riskestimatorforcross}) is non-positive if
\begin{align}
   \nonumber  r(t_i) \leq \frac{2(d - 2)}{2 - \frac{1}{1 + \rho}} = \frac{2(d - 2)}{1 + \frac{\rho}{1 + \rho}}.
\end{align}

This holds if $(d - 2) < b \leq \frac{2(d - 2)}{1 - \frac{\rho}{1 + \rho}}$, or equivalently, if 
\begin{align*}
    0 \leq \frac{\rho}{1 + \rho} \leq \frac{2(d - 2) - b}{b} = \frac{2(d - 2) }{b} - 1,
\end{align*}
or
\begin{align}
    0 \leq  \rho \leq \frac{\frac{2(d- 2)}{b} - 1}{1 - \left(\frac{2(d- 2)}{b} - 1\right)} = \frac{2(d - 2) - b}{2b - 2(d - 2)},
  \label{rho_bound_new}
  \end{align}
which completes the proof.
\end{proof}

Now for $s_0$ such that $\sqrt{s_0} \ge  \sqrt{t_0} + \sqrt{D}$, let\\
    \begin{equation} 
    \textbf{B} = \left\{x \in \mathbf{R}^d: \max\limits_{1\leq i \leq k} ||x - \theta_i||^2 = \max\limits_{1\leq i \leq k} t_i  \le  s_0\right\}.
    \label{BCondition}
     \end{equation}

\begin{lemma} \label{Lemma3.4}
For  $\textbf{A}$ in (\ref{A1Condition}) and $\textbf{B}$ in (\ref{BCondition}), $\textbf{A} \cup \textbf{B} = \mathbf{R}^d$.
\end{lemma}
\begin{proof}
For any pair $t_i, t_j$, it must be the case that at least one of the following holds:
\begin{enumerate}
    \item[(i)] both $t_i$ and $t_j$ are less than $s_0$, 
    \item[(ii)] both $t_i$ and $t_j$ are greater than $t_0$.
\end{enumerate}
To see this, note that if the lower value is less than $t_0$, the higher value cannot be greater than $s_0$. Also, if the larger value is greater than $s_0$, the smaller value is greater than $t_0$. More precisely, if $t_i < t_j$ and $t_i \leq t_0$, then from (\ref{triangle inequalities1}), $\sqrt{t_j} \leq \sqrt{t_i} + \sqrt{D_{ij}} \leq \sqrt{t_0} + \sqrt{D} \leq \sqrt{s_0}$.  On the other hand, if $t_i < t_j$ and $t_j \geq s_0$, then from (\ref{triangle inequalities2}), $\sqrt{t_i} \geq |\sqrt{t_j} - \sqrt{D_{ij}}| \geq  |\sqrt{s_0} - \sqrt{D_{ij}}| \geq \sqrt{t_0}$. Thus, $\textbf{A}$ and $\textbf{B}$ decompose $\mathbf{R}^d$ into two overlapping regions.
\end{proof}

Lemma \ref{LemmaA1brho} provides sufficient conditions to satisfy the conditions
for $\hat{\Delta}\delta_*(x) \le 0$ when $x \in \textbf{A}$, whereas Lemma \ref{Lemma3.1} provides sufficient conditions for $\hat{\Delta}\delta_*(x) \le 0$ when $x \in \textbf{B}$.   Combined with Lemma \ref{Lemma3.4}, these yield the following main result for establishing the minimaxity of a proper Bayes multiple shrinkage estimator.

\begin{theorem}\label{Theorem3.5}
   Suppose $r'(t) \ge 0$, $0 \leq r(t) \leq b = \sup r(\cdot)$ where $(d-2) < b < 2(d-2)$, and $r(t)/t$ is non-increasing in $t$ for $t>0$.  Then the corresponding multiple shrinkage estimator $\delta_{*}(x)$ in (\ref{MultipleShrinkageEstimator}) is minimax provided $r(t) \leq  (d-2)$ for all $t \leq s_0$ where $ \sqrt{s_0} \geq \sqrt{t_0} + \sqrt{D} = \frac{\sqrt{D}}{\rho} + \sqrt{D} = \sqrt{D}\left(1 + \frac{1}{\rho}\right)$ for $\rho \in \left(0, \frac{2(d - 2) - b}{2b - 2(d - 2)}\right]$.
    \label{MainResult}
\end{theorem}

This result  provides a viable framework, which we have illustrated in Figure \ref{fig Illustration}, for the construction of proper Bayes minimax multiple shrinkage estimators, offering clear instructions on how to achieve this goal.  In particular, note that to satisfy condition $r(t) \leq (d-2)$ for all $t \leq s_0$, modifications to the marginal functions may be necessary. In the upcoming section, we demonstrate that this requirement can be fulfilled by a suitable rescaling of the marginal $m(x)$. 

\begin{figure}[h]
    \centering    \includegraphics[width=1.0\linewidth,height=0.5\linewidth]{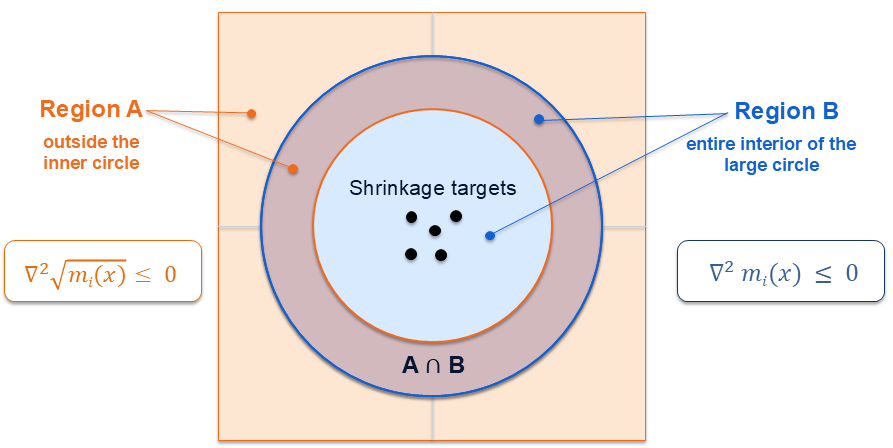}
    \caption{\textbf{Illustration of our framework:} We identify two overlapping regions in $\mathbf{R}^d$—Region A (orange), the outer region of a sphere with radius $\sqrt{t_0}$, and Region B (blue), the inner region of a sphere with radius $\sqrt{s_0}$. Within Region A, we leverage $\nabla^2 \sqrt{m_{\pi}(x)} \leq 0$, while in Region B, we leverage $\nabla^2 m_{\pi}(x) \leq 0$.}
    \label{fig Illustration}
\end{figure}

\section{Rescaling marginals}
\label{Scaled marginals section}

By rescaling the marginal function appropriately, we can ensure that the conditions of Theorem \ref{Theorem3.5} are satisfied so that the resulting multiple shrinkage estimator is minimax. This approach allows us to strike a balance between effective shrinkage and preserving the minimax optimality of the estimators.
This rescaling serves as a crucial step in our overall framework, contributing to the practical implementation and applicability of minimax multiple shrinkage estimators as well as furthering our understanding of the underlying principles behind their construction.

As before, we consider marginals $m(x)$ of the form $f(t)$, where $t = ||x||^2$. A rescaled version of $m(x)$ denoted by $m_a(x)$ is then obtained as $f_a(t) = \frac{1}{a^{d/2}} f(||x||^2/a)$, for which we have the following,  
\begin{align}
   \nonumber  f'_a(t) & = \frac{1}{a^{d/2 + 1}} f'\left( \frac{t}{a}\right) \\
   \nonumber  \text{ and }\; r_a(t) & = -\frac{2f'_a(t)t}{f_a(t)}  = -\frac{2f'(t/a)}{f(t/a)}\frac{t}{a} = r\left(\frac{t}{a}\right).
\end{align}
Note that for $a > 1$, such a rescaling serves to flatten an increasing $r(\cdot)$ function. Thus, to
satisfy the bounding condition on $r(\cdot)$ in Theorem \ref{MainResult}, we simply need to choose $a$ so that $r_a(s_0) = r(s_0/a) \le (d-2)$ or $a \ge \frac{s_0}{r^{-1}(d-2)}$.  Setting $s_0 = D\left(1 + \frac{1}{\rho}\right)^2$, we obtain our desired conditions for $a$.

\begin{theorem}
 Suppose $r'(t) > 0$, $0 \leq r(t) \leq b = \sup r(\cdot)$ is such that $(d-2) < b < 2(d-2)$, and $r(t)/t$ is non-increasing in $t$ for $t > 0$.  Then, the multiple shrinkage estimator $\delta_{*}(x)$ in (\ref{MultipleShrinkageEstimator}) constructed using the rescaled marginal $m_a(x)$ is minimax provided $a \geq \frac{D}{r^{-1}(d-2)}\left(1 + \frac{1}{\rho}\right)^2$  for  $\rho \in \left(0, \frac{2(d - 2) - b}{2b - 2(d - 2)}\right]$.
    \label{MainResult2}
\end{theorem}
This theorem establishes conditions under which the multiple shrinkage estimator constructed using a rescaled marginal $m_a(x)$ is minimax. It provides practical guidelines for selecting the parameter $a$ to ensure the minimax property of the estimator. By satisfying the conditions outlined in the theorem, practitioners can construct estimators that achieve the desired optimality guarantees while incorporating the necessary shrinkage effects.  

Moreover, Theorem \ref{MainResult2} also reveals the conditions under which the original, unrescaled marginals can yield minimax multiple shrinkage estimators.
\begin{corollary}
    Under the settings of Theorem \ref{MainResult2}, if $D \leq \frac{r^{-1}(d-2)}{\left(1 + \frac{1}{\rho}\right)^2}$, then the multiple shrinkage estimator $\delta_{*}(x)$ in (\ref{MultipleShrinkageEstimator}) constructed using $m_a(x)$ with $a =1$, i.e. the unrescaled marginal $m(x)$, is minimax.
\end{corollary}

\begin{example}[A pseudo-Bayes multiple shrinkage estimator which is minimax]
\label{An example of a pseudo-Bayes multiple shrinkage estimator which is minimax}
Now we present an example of a multiple shrinkage estimator constructed using a pseudo-marginal $m(x) = f(||x||^2)$ with $f(t) = \left( \frac{1}{1 \; + \;t} \right) ^\frac{b}{2}$, where $(d-2) < b < 2(d-2)$. We note that for such marginal densities, $f'(t) = \frac{-b}{2}\left( \frac{1}{1 \; + \;t} \right) ^{\left(\frac{b}{2} + 1\right)}$ and $r(t) = b\frac{t}{1+t} < b < 2(d-2)$ (see Fig. \ref{fig:pseucomarginalandr}). Additionally, $r(t)/t$ is decreasing in $t$. Also, we have 
$s_0 = r^{-1}(d - 2) = \frac{(d-2)}{b - (d-2)}$. Suppose we have $k$ shrinkage targets $\theta_1,\ldots,\theta_k$. Let $m_i(x) = f(t_i)$, where $t_i = ||x - \theta_i||^2$ for $i = 1,\ldots,k$. Now, consider a combined marginal function rescaled by $a$, $ m_{*a}(x) = \sum\limits_{i=1}^{k}\;w_i\;m_{ai}(x) =\sum\limits_{i=1}^{k}\;w_i\; \frac{1}{a^{d/2}} \left(1 \; + \;\frac{||x - \theta_i||^2}{a} \right) ^{-\frac{b}{2}}$. Thus, using Theorem \ref{MainResult2}, the resulting multiple shrinkage estimator $ \delta_{*a}(x) = x + \frac{\nabla m_{*a}(x)}{m_{*a}(x)}$ will be minimax for  $\rho \in \left(0, \frac{2(d - 2) - b}{2b - 2(d - 2)}\right]$ and scaling $a \geq \frac{D}{r^{-1}(d-2)}\left(1 + \frac{1}{\rho}\right)^2$. By choosing an appropriate rescaling parameter $a$ based on the given conditions, the constructed multiple shrinkage estimator achieves minimaxity.

\begin{figure}
    \centering
    \includegraphics[scale=0.6]{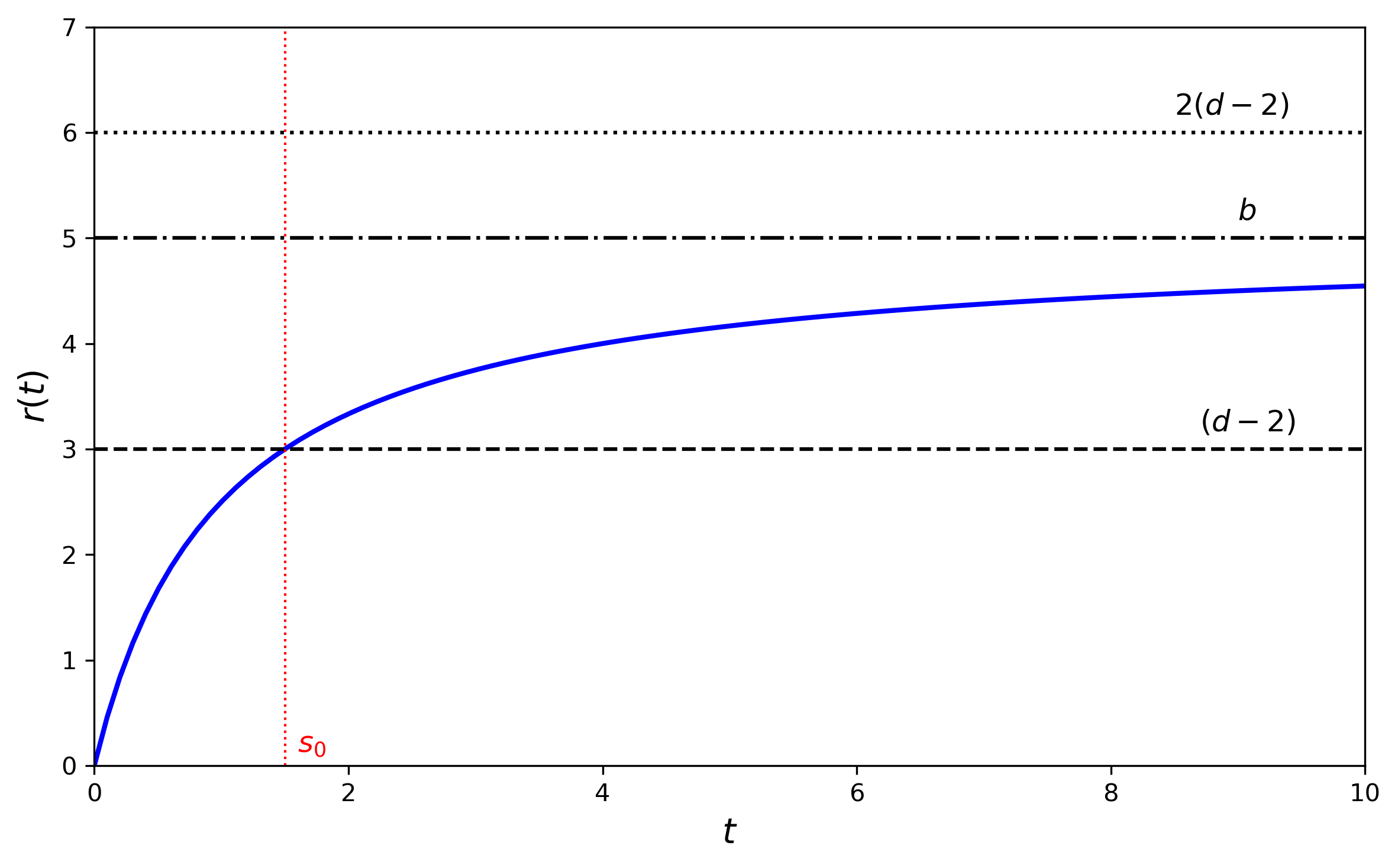}
    \caption{Plot of the $r(\cdot)$ function corresponding to the marginal density $m(x) = f(t)$ with $f(t) = \left( \frac{1}{1 \; + \;t} \right) ^\frac{b}{2}$ for $d = 5$ and $b = 5$. }
    \label{fig:pseucomarginalandr}
\end{figure}
    
\end{example}

Finally it is interesting to comment on the role of the rescaling constant $a$ for ensuring that the associated multiple shrinkage estimator $\delta_*(x)$ in (\ref{MultipleShrinkageEstimator}) with respect to the rescaled marginal becomes minimax. The minimum threshold for such an $a$ from Theorem \ref{MainResult2}, namely 
\begin{align}
    a \geq \frac{D}{r^{-1}(d-2)}\left(1 + \frac{1}{\rho}\right)^2,
    \label{a_cutoff_new}
\end{align}
suggests that 
the needed rescaling is minimal if $D \approx r^{-1}(d-2)$. Recall that, whenever $r(t) \leq (d - 2)$ implies $\nabla^2 m(x) \leq 0$.  This means that, if the original marginal has the superharmonic property ($\nabla^2 m(x) \leq 0$) over $\textbf{B}$ in (\ref{BCondition}), the region which covers all the target candidates, we will require a small scaling (at least theoretically) to get a minimax multiple shrinkage estimator. In other words, we are able to construct proper priors which are flat enough to cover the target candidates and achieve minimaxity.

\section{Minimax Multiple Shrinkage Estimators under Normal Variance Mixture Priors}
\label{Adjusting priors yielding scaled marginals}
Some of the best-known proper Bayes minimax shrinkage estimators such as the Strawderman estimator are induced by normal variance mixture priors \cite{FSW1998}, \cite{s1971}.  For the construction of proper Bayes minimax multiple shrinkage estimators that combine such estimators, we now proceed to show how the marginals under this class of priors can be rescaled via a simple augmentation adjustment of the commonly used prior parametrization, to satisfy all the minimaxity conditions on $r(\cdot)$ in Theorem \ref{Theorem3.5}.

With this common parametrization, these priors are obtained as
\begin{align} \label{initialparametrization}
  \theta | \lambda & \sim N_d\left(0, \frac{1 - \lambda}{\lambda}I_d\right)\:, \;\;\; \lambda \sim h\:,
\end{align}
where $h(\cdot)$ is a density over $(0,1]$. Leading to the following marginal distribution for $X$
\begin{align}
   \nonumber  X | \lambda & \sim N_d\left(0, \frac{1 }{\lambda}I_d\right), \;\;\; \lambda \sim h\:,
\end{align}
the marginal density $m(x) = f_{h}(||x||^2)$ can then be expressed as\hspace{-1mm}:
\begin{align}
   \nonumber  f_h(t) &  = \frac{1}{(2\pi)^{d/2}}\int\limits_{0}^1 \lambda^{d/2} h(\lambda) e^{-\frac{\lambda t}{2}} \text{d}\lambda\;, \\
   \nonumber \text{with } f_h^{'}(t) & = -\frac{1}{2}\frac{1}{(2\pi)^{d/2}}\int\limits_{0}^1 \lambda^{d/2 + 1} h(\lambda) e^{-\frac{\lambda t}{2}} \text{d}\lambda\;.
\end{align}

\cite{s1971} provided sufficient conditions under which $0 \leq r(t) \leq 2(d-2)$ for the unconditional density $h(\lambda) = (1-\alpha)\lambda^{-\alpha} $ for $0 < \lambda \leq 1$, where $0 \leq \alpha < 1$. \cite{FSW1998} considered more general densities $h(\cdot)$ and obtained conditions on priors which lead to proper Bayes minimax estimators. The following Lemma establishes some of the similar results which are needed to satisfy the Theorem \ref{Theorem3.5} conditions for the construction of proper Bayes multiple shrinkage estimators.

\begin{lemma}
\label{lemma51}

Let $h:(0,1)\to(0,\infty)$ be an integrable function that is
locally absolutely continuous. Define $\ell(\lambda) = -\frac{\lambda h'(\lambda)}{h(\lambda)}$ for almost every $\lambda\in(0,1)$. Suppose that:

\begin{enumerate}
    \item $\ell$ is non-decreasing and $\ell(\lambda)\geq A$ almost everywhere;

    \item $ \int_0^1
    \lambda^{d/2}h(\lambda)|\ell(\lambda)|\,d\lambda
    <\infty;$

    \item $\lim_{\lambda\downarrow0}
    \lambda^{d/2+1}h(\lambda)=0;$

    \item the finite limit $ h_1:=\lim_{\lambda\uparrow1}h(\lambda)$ exists.
\end{enumerate}
Then the associated function $r(t)$ is non-decreasing on
$[0,\infty)$, $r(t)/t$ is non-increasing on $(0,\infty)$, and $0\leq r(t)\leq d+2-2A.$ Consequently, if $d+2-2A\leq b$, then $0\leq r(t)\leq b.$
\end{lemma}

\begin{proof}
    We have 
    \begin{align}
    \nonumber r(t) & = \frac{-2f_h^{'}(t)t}{f_h(t)}  =  t\;\frac{\int\limits_{0}^{1} \lambda^{d/2 + 1}\;e^{-\frac{t\lambda}{2}}\; h(\lambda)\text{d}\lambda}{\int\limits_{0}^{1} \lambda^{d/2}\;e^{-\frac{t\lambda}{2}}\; h(\lambda)\text{d}\lambda}\;. 
\end{align}
On applying integration  by parts in the numerator, we get
\begin{align}
   \nonumber  \int\limits_{0}^{1} \lambda^{d/2 + 1}\;e^{-\frac{t\lambda}{2}}\; h(\lambda)\text{d}\lambda & = \frac{2}{t}\;\frac{d + 2}{2}\;\int\limits_{0}^{1} \lambda^{d/2 }\;e^{-\frac{t\lambda}{2}}\; h(\lambda)\text{d}\lambda \\
   \nonumber & + \frac{2}{t}\;\int\limits_{0}^{1} \frac{\lambda h'(\lambda)}{h(\lambda)} \lambda^{d/2}\;e^{-\frac{t\lambda}{2}}\; h(\lambda)\text{d}\lambda \\
    \nonumber & - \frac{2}{t}\;\left[\lambda ^{d/2 + 1}\;e^{-\frac{t\lambda}{2}}\; h(\lambda) \right]_{0}^{1}\;.
\end{align}
We thus obtain 
\begin{align}
    \nonumber r(t) & = (d + 2 )\; + 2\mathbf{E}^{*}_t\left[\frac{\lambda h'(\lambda)}{h(\lambda)}\right] - 2\; \frac{\lim_{\lambda \to 1} \;\lambda ^{d/2 + 1}\;e^{-\frac{t\lambda}{2}}\; h(\lambda)}{\int\limits_{0}^{1} \lambda^{d/2}\;e^{-\frac{t\lambda}{2}}\; h(\lambda)\text{d}\lambda} + 2\;\frac{\lim_{\lambda \to 0}\;\lambda ^{d/2 + 1}\;e^{-\frac{t\lambda}{2}}\; h(\lambda)}{\int\limits_{0}^{1} \lambda^{d/2}\;e^{-\frac{t\lambda}{2}}\; h(\lambda)\text{d}\lambda}, \\
    \nonumber & = (d + 2 )\; - 2\mathbf{E}^{*}_t [\ell(\lambda)] - 2\; \frac{\lim_{\lambda \to 1} \;\lambda ^{d/2 + 1}\;e^{-\frac{t\lambda}{2}}\; h(\lambda)}{\int\limits_{0}^{1} \lambda^{d/2}\;e^{-\frac{t\lambda}{2}}\; h(\lambda)\text{d}\lambda} , 
\end{align}
since $\lim\limits_{\lambda \to 0}\lambda^{d/2+1}\; h(\lambda) = 0 $ and  the expectation $E^{*}_t$ is with respect to the density $ g_t(\lambda) =\frac{\lambda^{d/2}\;e^{-\frac{t\lambda}{2A}}\; h(\lambda)}{\int\limits_{0}^{1} \lambda^{d/2}\;e^{-\frac{t\lambda}{2A}}\; h(\lambda)\text{d}\lambda}$ on  $0 < \lambda < 1$ for fixed $t$. This family has the decreasing monotone likelihood ratio property with parameter $t$.  Since the function $\ell(\lambda) $ is non-decreasing in $\lambda$, $\mathbf{E}^{*}_t [\ell(\lambda)]$ is decreasing in $t$. Also, $\frac{\lim_{\lambda \to 1} \;\lambda ^{d/2 + 1}\;e^{-\frac{t\lambda}{2}}\; h(\lambda)}{\int\limits_{0}^{1} \lambda^{d/2}\;e^{-\frac{t\lambda}{2}}\; h(\lambda)\text{d}\lambda}$ is equal to $\frac{h_1}{\int\limits_{0}^{1} \lambda^{d/2}\;e^{\frac{t(1 - \lambda)}{2}}\; h(\lambda)\text{d}\lambda}$ which is decreasing in $t$.  Hence, the function $r(t)$ is increasing in $t$. 
Furthermore, 
\begin{align}
  \nonumber  r(t) \leq (d + 2 )\; - 2\mathbf{E}^{*}_t [\ell(\lambda) ]\leq ( d + 2) - 2A \leq b.
\end{align}
On the other hand,
\begin{equation}
    \nonumber \frac{r(t)}{t}  = \frac{-2f_h^{'}(t)}{f_h(t)} = \frac{\int\limits_{0}^1 \lambda^{d/2 + 1} h(\lambda) e^{-\frac{\lambda t}{2}} \text{d}\lambda}{\int\limits_{0}^1 \lambda^{d/2} h(\lambda) e^{-\frac{\lambda t}{2}} \text{d}\lambda} = \mathbf{E}^*_{t}[\lambda]\;.
\end{equation}
Finally, by the monotone likelihood ratio property of densities $g_t(\cdot)$,
$\mathbf{E}^*_{t}[\lambda]$, and thus $r(t)/t$, is non-increasing in $t$.
\end{proof}

Now we consider a prior $\pi_a$ on $\theta$ that facilitates the rescaling of marginals. This prior distribution is obtained by augmenting $\lambda$ to $\lambda/a$ with $a \geq 1$, in (\ref{initialparametrization}) to obtain
\begin{align}
   \theta | \lambda & \sim N_d\left(0, \frac{1 - \lambda/a}{\lambda/a}I_d\right)\;, \;\;\; \lambda \sim h\:.
\end{align}
This leads to 
\begin{align}
   \nonumber  X | \lambda & \sim N_d\left(0, \frac{a }{\lambda}I_d\right)\;, \;\;\; \lambda \sim h\:,
\end{align}
so that the marginal density $m_a(x) = f_{h,a}(||x||^2)$ can be expressed as\hspace{-1mm}:
\begin{align}
   \nonumber  f_{h,a}(t) &  = \frac{1}{(2\pi a)^{d/2}}\int\limits_{0}^1 \lambda^{d/2} h(\lambda) e^{-\frac{\lambda t}{2a}} \text{d}\lambda = \frac{1}{a^{d/2}} f_h\left(\frac{t}{a}\right).
\end{align}
We observe that $f_{h,a}'(t)$ is equal to $\frac{1}{a^{d/2+1}}f_{h}'(t/a)$, and similarly that $r_a(t) = r(t/a)$. 
Hence, if $m(x)$ leads to a generalized (or proper for $d \geq 5$) Bayes estimator which is minimax, then choosing $a \geq \max\left\{1, \left(1 + \frac{1}{\rho}\right)^2\frac{D}{r^{-1}(d-2)}\right\}$, $m_a(x)$ yields a generalized (or proper for $d \geq 5$) Bayes  multiple shrinkage estimator which is minimax.
 Thus, we get the following result\hspace{-1mm}:

\begin{theorem}
\label{thm:properBayesmultipleshrinkage}
Suppose that $h$ satisfies all the assumptions of
Lemma~\ref{lemma51} with
$\ell(\lambda)\geq A$ almost everywhere. Let $b$ satisfy $d+2-2A\leq b, \; d-2<b<2(d-2).$ Let $t^\star=\sup\{t>0:r(t)\leq d-2\}.$ Then the Bayes multiple-shrinkage estimator under the prior $\pi_{*a}(\theta)
=
\sum_{i=1}^k w_i\pi_a(\theta-\theta_i)$ is minimax provided $a\geq
\max\left\{
1,\,
\frac{D}{t^\star}
\left(1+\frac1\rho\right)^2
\right\},$ where $0<\rho<1,
\;
\rho\leq
\frac{2(d-2)-b}{2b-2(d-2)}.$
\end{theorem}

\begin{example}[Rescaled Strawderman priors]
\label{ex:rescaled-strawderman}

Let $h_\alpha(\lambda) = (1-\alpha)\lambda^{-\alpha},
\; 0<\lambda<1,\;0\leq \alpha<1.$ Then $h_\alpha$ is a proper probability density on $(0,1)$.
Consider the rescaled hierarchical prior
\[
\theta\mid\lambda
\sim
N_d\left(
0,
\left(\frac{a}{\lambda}-1\right)I_d
\right),
\qquad
\lambda\sim h_\alpha,
\]
where $a\geq1$. Let $\pi_a$ denote the resulting marginal prior
on $\theta$. For this mixing density, $\ell_\alpha(\lambda)
=
-\frac{\lambda h_\alpha'(\lambda)}
       {h_\alpha(\lambda)}
=
\alpha.$ Hence, $\ell_\alpha$ is nondecreasing and is bounded below by
$A=\alpha$. Moreover,
\[
\int_0^1
\lambda^{d/2}h_\alpha(\lambda)
|\ell_\alpha(\lambda)|\,d\lambda
<\infty,\; \lim_{\lambda\downarrow0}
\lambda^{d/2+1}h_\alpha(\lambda)=0,
\;
\lim_{\lambda\uparrow1}h_\alpha(\lambda)=1-\alpha.
\]
Thus, all the conditions of Lemma~\ref{lemma51}
are satisfied.

It follows that the corresponding shrinkage function
$r_\alpha(t)$ is nondecreasing, $r_\alpha(t)/t$ is
nonincreasing, and $0\leq r_\alpha(t)\leq b_\alpha,
\; b_\alpha=d+2-2\alpha.$ In fact, $\lim\limits_{t\to\infty}r_\alpha(t)=b_\alpha.$ For the multiple-shrinkage construction, we require $d-2<b_\alpha<2(d-2).$ Since $0\leq\alpha<1$, the first inequality is automatic, while
the second is equivalent to $\alpha>3-\frac d2.$ Thus, the allowable values of $\alpha$ satisfy $0\leq\alpha<1,
\; \alpha>3-\frac d2.$
In particular, this requires $\alpha>1/2$ when $d=5$,
$\alpha>0$ when $d=6$, and imposes no restriction beyond
$0\leq\alpha<1$ when $d\geq7$. The allowable range of $\rho$ is
\[
0<\rho<1,
\qquad
\rho\leq
\rho_{\max}(\alpha)
:=
\frac{2(d-2)-b_\alpha}
     {2b_\alpha-2(d-2)}
=
\frac{d-6+2\alpha}{4(2-\alpha)}.
\]
Equivalently, for a fixed $\rho\in(0,1)$, this condition may be written as $\alpha
\geq
\frac{8\rho-d+6}{2+4\rho},$ together with $0\leq\alpha<1$.

Define $t_\alpha = \inf\left\{
t>0:r_\alpha(t)\geq d-2
\right\}.$ Since $r_\alpha$ increases from $0$ to
$b_\alpha>d-2$, the quantity $t_\alpha$ is finite. It is the
solution of $\int_0^1
\lambda^{d/2-\alpha}
\exp\left\{
\frac{t_\alpha(1-\lambda)}{2}
\right\}
\,d\lambda
=
\frac{1}{2-\alpha}.$

Let $\pi_{*a}(\theta)
=
\sum\limits_{i=1}^k
w_i\pi_a(\theta-\theta_i),
\;
w_i>0,
\;
\sum\limits_{i=1}^k w_i=1,$ and let $D
=
\max_{i\neq j}
\|\theta_i-\theta_j\|^2.$ Then the Bayes multiple shrinkage estimator induced by
$\pi_{*a}$ is minimax provided
\[
a
\geq
\max\left\{
1,\,
\frac{D}{t_\alpha}
\left(1+\frac{1}{\rho}\right)^2
\right\},
\]
where $0<\rho<1,
\;
\rho\leq
\frac{d-6+2\alpha}{4(2-\alpha)}.$ Because $h_\alpha$ is a proper mixing density and $a\geq1$
ensures that $(a/\lambda-1)I_d$ is nonnegative for every
$\lambda\in(0,1)$, $\pi_{*a}$ is a proper prior. Consequently,
the resulting estimator is both proper Bayes and minimax.
\end{example}

\begin{example}[Rescaled multivariate Student priors]
Consider the mixing density
\[
h_{m,\kappa}(\lambda)
=
\frac{(m\kappa/2)^{m/2}}{\Gamma(m/2)}
\lambda^{m/2-1}(1-\lambda)^{-(m+2)/2}
\exp\left\{
-\frac{m\kappa\lambda}{2(1-\lambda)}
\right\},
\qquad 0<\lambda<1,
\]
where $m>0$ and $\kappa>0$. Under the hierarchical prior
\[
\theta\mid\lambda
\sim
N_d\left(0,\frac{1-\lambda}{\lambda}I_d\right),
\qquad
\lambda\sim h_{m,\kappa},
\]
the marginal prior on $\theta$ is the multivariate Student distribution
\[
\pi(\theta)
=
\frac{\Gamma((m+d)/2)}
{\Gamma(m/2)(m\pi\kappa)^{d/2}}
\left(
1+\frac{\|\theta\|^2}{m\kappa}
\right)^{-(m+d)/2}.
\]

For this mixing density,
\[
\ell(\lambda)
=
-\frac{\lambda h'_{m,\kappa}(\lambda)}
{h_{m,\kappa}(\lambda)}
=
1-\frac{m}{2}
-\frac{(m+2)\lambda}{2(1-\lambda)}
+\frac{m\kappa\lambda}{2(1-\lambda)^2},
\]
and
\[
\ell'(\lambda)
=
\frac{
m\kappa-(m+2)
+
\{m\kappa+m+2\}\lambda
}
{2(1-\lambda)^3}.
\]
Hence, if $\kappa\geq\frac{m+2}{m},$ then $\ell$ is non-decreasing and $\ell(\lambda)\geq A :=1-\frac{m}{2}.$ We next verify the remaining assumptions of
Lemma~\ref{lemma51}. Near zero, $\lambda^{d/2}h_{m,\kappa}(\lambda)|\ell(\lambda)|
=
O\left(
\lambda^{(d+m)/2-1}
\right),$ which is integrable because $d+m>0$. Near one, $|\ell(\lambda)|
=
O\left((1-\lambda)^{-2}\right),$ and therefore
\[
\lambda^{d/2}h_{m,\kappa}(\lambda)|\ell(\lambda)|
=
O\left[
(1-\lambda)^{-(m+6)/2}
\exp\left\{
-\frac{m\kappa}{2(1-\lambda)}
\right\}
\right],
\]
which is integrable because the exponential term dominates the
polynomial singularity. Thus,
\[
\int_0^1
\lambda^{d/2}
h_{m,\kappa}(\lambda)
|\ell(\lambda)|
\,d\lambda
<\infty.
\]

Moreover, $\lambda^{d/2+1}h_{m,\kappa}(\lambda)
\asymp
C_{m,\kappa}
\lambda^{(d+m)/2}
\longrightarrow0
\; \text{as }\lambda\downarrow0,$ and $h_1
:=
\lim_{\lambda\uparrow1}h_{m,\kappa}(\lambda)
=
0.$
Hence, all the regularity, integrability, and endpoint assumptions of Lemma~\ref{lemma51} are satisfied.

It follows that the corresponding shrinkage function $r(t)$ is
nondecreasing, $r(t)/t$ is nonincreasing, and $0\leq r(t)\leq d+2-2A=d+m.$ In fact, this upper bound is exact. To see this, let $I_j(t)
=
\int_0^1
\lambda^{d/2+j}
h_{m,\kappa}(\lambda)
e^{-t\lambda/2}
\,d\lambda,
\;  j=0,1,$ 
so that $r(t)=t\frac{I_1(t)}{I_0(t)}.$ Since $h_{m,\kappa}(\lambda)
\asymp
C_{m,\kappa}\lambda^{m/2-1}
\; \text{as }\lambda\downarrow0,$
standard Laplace asymptotics give, with $\nu=\frac{d+m}{2},$ $I_0(t)
\sim
C_{m,\kappa}\Gamma(\nu)
\left(\frac{2}{t}\right)^\nu$
and $I_1(t)
\sim
C_{m,\kappa}\Gamma(\nu+1)
\left(\frac{2}{t}\right)^{\nu+1}.$
Consequently,
\[
\lim_{t\to\infty}r(t)
=
\lim_{t\to\infty}
t\frac{I_1(t)}{I_0(t)}
=
2\frac{\Gamma(\nu+1)}{\Gamma(\nu)}
=
2\nu
=
d+m.
\]
Since $r$ is nondecreasing, $\sup_{t\geq0}r(t)=d+m.$ Thus, in the notation of Theorem~\ref{thm:properBayesmultipleshrinkage},
we may take $b=d+m.$ For the multiple-shrinkage construction, the condition $d-2<b<2(d-2)$ therefore becomes $d-2<d+m<2(d-2),$ which is equivalent, for $m>0$, to $0<m<d-4.$ For integer degrees of freedom, this may be written as $1\leq m\leq d-5,$ and hence the construction requires $d\geq6$.

The allowable values of $\rho$ satisfy
\[
0<\rho<1,
\qquad
\rho\leq
\frac{2(d-2)-(d+m)}
     {2(d+m)-2(d-2)}
=
\frac{d-m-4}{2(m+2)}.
\]

Define the generalized inverse $t_{m,\kappa}
=
\sup\left\{
t>0:r(t)\leq d-2
\right\}.$ Because $r(0)=0$ and $\lim\limits_{t\to\infty}r(t)=d+m>d-2,$ the quantity $t_{m,\kappa}$ is finite. Let $\pi_{*a}(\theta)
=
\sum_{i=1}^k
w_i\pi_a(\theta-\theta_i),
\;
w_i>0,
\;
\sum_{i=1}^k w_i=1,$
and let $D
=
\max_{i\neq j}
\|\theta_i-\theta_j\|^2.$
Then the Bayes multiple-shrinkage estimator induced by
$\pi_{*a}$ is minimax provided
\[
a
\geq
\max\left\{
1,\,
\frac{D}{t_{m,\kappa}}
\left(1+\frac1\rho\right)^2
\right\},
\]
where $0<\rho<1,
\;
\rho\leq
\frac{d-m-4}{2(m+2)}.$ Because $h_{m,\kappa}$ is a proper mixing density and $a\geq1$
ensures that $\left(\frac a\lambda-1\right)I_d$ is nonnegative for every $\lambda\in(0,1)$, the mixture
$\pi_{*a}$ is a proper prior. Consequently, the corresponding
multiple shrinkage estimator is both proper Bayes and minimax.
\end{example}

\section{Simulation Study}
\label{Simulation Study}
In this section, we investigate the performance of two-target multiple-shrinkage estimators induced by mixtures of rescaled Strawderman priors, as constructed in Example~\ref{ex:rescaled-strawderman}. We compare the unrescaled estimator, corresponding to $a=1$, and its rescaled versions with the associated single-target Strawderman estimators. The simulations illustrate both the ability of the multiple-shrinkage estimator to adapt between competing targets and the effect of the scaling parameter on its risk.

We consider Strawderman mixing densities $h_\alpha(\lambda) = (1-\alpha)\lambda^{-\alpha},
\; 0<\lambda<1, \; 0\leq\alpha<1.$ From example \ref{ex:rescaled-strawderman}, the
radial marginal kernel, up to a positive multiplicative constant, is
\[
m_{\alpha,a}(t)
\propto
a^{-d/2}
\int_0^1
\lambda^{d/2-\alpha}
\exp\left\{
-\frac{\lambda t}{2a}
\right\}
\,d\lambda.
\]
The unscaled shrinkage function is
\begin{align*}
    r_\alpha(t)
=
t
\frac{\int_0^1
\lambda^{d/2+1-\alpha}
e^{-\lambda t/2}
\,d\lambda
}{\int_0^1
\lambda^{d/2-\alpha}
e^{-\lambda t/2}
\,d\lambda} =
d+2-2\alpha
-
\frac{
2e^{-t/2}(t/2)^{d/2+1-\alpha}
}{
\gamma\left(
d/2+1-\alpha,t/2
\right)
},
\end{align*}
where $\gamma(s,x)
=
\int_0^x u^{s-1}e^{-u}\,du$ denotes the lower incomplete gamma function. The rescaled shrinkage function is $r_{\alpha,a}(t)
=
r_\alpha(t/a).$ 

For two targets $\theta_1,\theta_2\in\mathbb{R}^d$, define $t_i(x)=\|x-\theta_i\|^2,
\; i=1,2.$ The corresponding single-target estimators are
\[
\delta_{i,a}(x)
=
x-
\frac{r_\alpha(t_i(x)/a)}
     {t_i(x)}
(x-\theta_i),
\qquad i=1,2,
\]
with the value at $t_i(x)=0$ defined by continuous extension.
For equal prior component weights, the posterior mixture weights are
\[
\rho_{i,a}(x)
=
\frac{
m_{\alpha,a}(t_i(x))
}{
m_{\alpha,a}(t_1(x))
+
m_{\alpha,a}(t_2(x))
},
\qquad i=1,2,
\]
and the resulting multiple-shrinkage estimator is $\delta_a^*(x)
=
\rho_{1,a}(x)\delta_{1,a}(x)
+
\rho_{2,a}(x)\delta_{2,a}(x).$

We place the targets symmetrically on the line spanned by
$\mathbf{1}_d=(1,\ldots,1)^\top$: $\theta_1
=
-\frac{1}{2}\sqrt{\frac{D}{d}}\,\mathbf{1}_d,
\; 
\theta_2
=
\frac{1}{2}\sqrt{\frac{D}{d}}\,\mathbf{1}_d,$
so that $\|\theta_1-\theta_2\|^2=D.$ Parameter values on the line joining the targets are written as $\theta=c\mathbf{1}_d,$ where $c$ is a signed scalar coordinate. Notice that $\frac{\|\theta\|}{\sqrt d}=|c|$.

All reported Monte Carlo calculations use $N=10,000$ independent draws $X_s\sim N_d(\theta,I_d),
\;s=1,\ldots,N.$ Using the same observations for all estimators at a given parameter value, we estimate the risk difference relative to the maximum
likelihood estimator by
\[
\widehat{\Delta}(\theta,\delta)
=
\frac1N
\sum_{s=1}^N
\left[
\|\delta(X_s)-\theta\|^2
-
\|X_s-\theta\|^2
\right].
\]

We consider squared target separations $D\in\{100,200,500\}$ when $d=6,$ and $D\in\{500,1000\}$ when $d=10.$ Parameter values on the line joining and extending beyond the two targets are written as $\theta=c\mathbf{1}_d,$ where $c$ is a signed scalar coordinate. For each combination of $d$ and $D$, we evaluate the risk at $c
\in
\sqrt{\frac{D}{d}}
\left\{
-1,-0.75,-0.50,-0.25,0,
0.25,0.50,0.75,1
\right\}.$ We consider $a\in\{1,3,5,10,12\}$ and $\alpha\in\{0.1,0.3,0.5,0.7,0.9\}.$ The choice $a=1$ corresponds to the unrescaled estimator, whereas $a>1$ gives its rescaled versions.

We first examine the unrescaled multiple-shrinkage estimator
$\delta_1^*$, corresponding to $a=1$, over a range of dimensions,
target separations, and values of the mixing parameter $\alpha$.
For each configuration, we evaluate the estimated risk along the
signed target-axis grid and compare it with the constant
MLE risk $d$.

For some combinations of $d$, $D$, and $\alpha$, the estimated risk
of $\delta_1^*$ remains below $d$ at every parameter value on the
evaluated grid. For example, Figure~\ref{fig:unscaled-d6-D40}
considers $d=6,\; D=40,\; \alpha=0.5.$ In this setting, the unrescaled multiple-shrinkage estimator adapts
between the two component estimators and achieves substantial risk reduction near both targets. No positive excess risk is observed on
the evaluated target-axis grid.

\begin{figure}[t]
    \centering
    \includegraphics[width=0.72\linewidth]
    {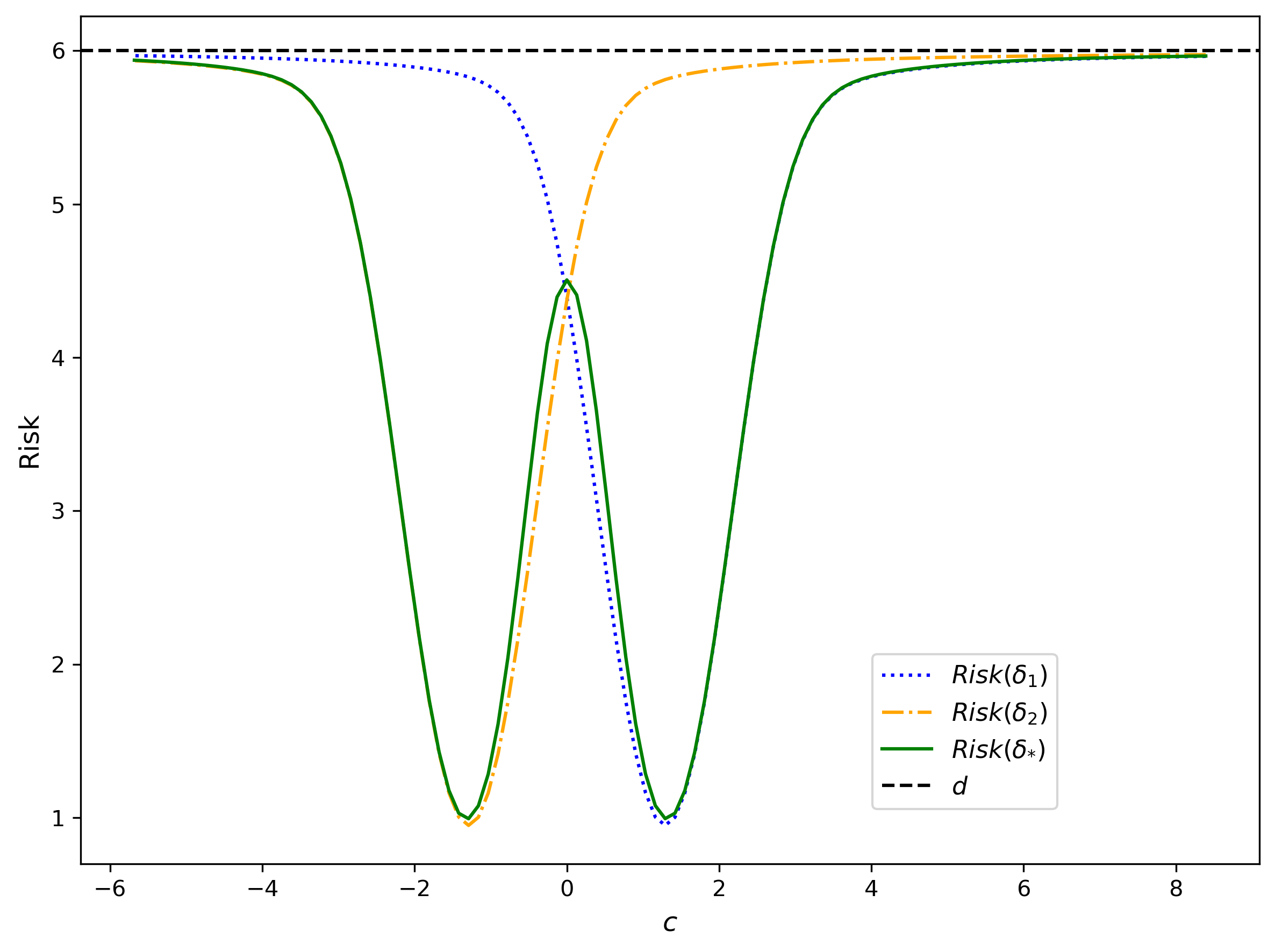}
    \caption{
    Estimated risk profiles of the two single-target Strawderman
    estimators and the corresponding unrescaled multiple-shrinkage
    estimator for $d=6$, $\alpha=0.5$, and $D=40$.
    The signed horizontal coordinate is defined by
    $\theta=c\mathbf{1}_d$, where
    $c=\theta^\top\mathbf{1}_d/d$.
    The horizontal dashed line denotes the constant MLE risk $d=6$.
    }
    \label{fig:unscaled-d6-D40}
\end{figure}

The numerical screening also suggests that larger target separations are generally required before a positive excess risk is observed in higher dimensions. For $d=10$ and $D=100$, for example, no positive excess was observed on the evaluated grid for any $\alpha\in\{0.1,0.3,0.5,0.7,0.9\}.$ This motivated the use of the larger separations $D\in\{100,200,500\}
\; \text{for }d=6,$ and $D\in\{500, 1000\}
\; \text{for }d=10$.

Tables~\ref{tab:unscaled-d6} and
\ref{tab:unscaled-d10} summarize the finite-grid screening results.
A checkmark indicates that no positive estimated excess risk was
observed over the evaluated target-axis grid. A cross indicates that
the estimated risk exceeded $d$ at one or more grid points.

\begin{table}[h]
\centering
\begin{minipage}[t]{0.47\linewidth}
\centering
\caption{
Finite-grid screening of the unrescaled estimator $\delta_1^*$ for
$d=6$. A checkmark indicates that no positive estimated excess risk
was observed on the evaluated target-axis grid; a cross indicates
that the estimated risk exceeded $d$ at one or more grid points.
}
\label{tab:unscaled-d6}
\begin{tabular}{c|ccccc}
\hline
& \multicolumn{5}{c}{$\alpha$} \\
$D$
& $0.1$
& $0.3$
& $0.5$
& $0.7$
& $0.9$
\\
\hline
$10$   & \checkmark & \checkmark & \checkmark & \checkmark & \checkmark \\
$40$   & \checkmark & \checkmark & \checkmark & \checkmark & \checkmark \\
$100$  & \ding{55}  & \ding{55}  & \ding{55}  & \ding{55}  & \checkmark \\
$200$  & \ding{55}  & \ding{55}  & \ding{55}  & \ding{55}  & \ding{55} \\
$500$  & \ding{55}  & \ding{55}  & \ding{55}  & \ding{55}  & \ding{55} \\
$1000$ & \ding{55}  & \ding{55}  & \ding{55}  & \ding{55}  & \ding{55} \\
\hline
\end{tabular}
\end{minipage}
\hfill
\begin{minipage}[t]{0.47\linewidth}
\centering
\caption{
Finite-grid screening of the unrescaled estimator $\delta_1^*$ for
$d=10$. The symbols have the same interpretation as in
Table~\ref{tab:unscaled-d6}.
}
\label{tab:unscaled-d10}
\begin{tabular}{c|ccccc}
\hline
& \multicolumn{5}{c}{$\alpha$} \\
$D$
& $0.1$
& $0.3$
& $0.5$
& $0.7$
& $0.9$
\\
\hline
$10$   & \checkmark & \checkmark & \checkmark & \checkmark & \checkmark \\
$40$   & \checkmark & \checkmark & \checkmark & \checkmark & \checkmark \\
$100$  & \checkmark & \checkmark & \checkmark & \checkmark & \checkmark \\
$200$  & \checkmark & \checkmark & \checkmark & \checkmark & \checkmark \\
$500$  & \ding{55}  & \ding{55}  & \ding{55}  & \checkmark & \checkmark \\
$1000$ & \ding{55}  & \ding{55}  & \ding{55}  & \ding{55}  & \ding{55} \\
\hline
\end{tabular}
\end{minipage}
\end{table}

We next examine configurations for $d=6$ and $d=10$ in which the
target separation $D$ is large enough that the unrescaled
multiple-shrinkage estimator $\delta_1^*$ exhibits a positive
estimated excess risk at one or more evaluated parameter values on the
target-axis grid. In these settings, rescaling can substantially
reduce the excess risk.

Detailed risk summaries are reported in
Tables~\ref{tab:risks_d10_alpha0.1}--\ref{tab:risks_d10_alpha0.9} for
$d=10$, and in
Tables~\ref{tab:risks_d6_alpha0.1}--\ref{tab:risks_d6_alpha0.9} for
$d=6$. Figure~\ref{fig:scaling-illustration} illustrates the effect
of the scaling parameter in the representative case $d=6,\; \alpha=0.1,\; D=100.$

Two clear patterns emerge from these numerical results. First, larger
target separations generally require larger values of the scaling
parameter $a$ before the observed excess risk on the evaluated grid is
substantially reduced. Second, for a fixed separation $D$, the excess
risk of the unrescaled estimator $\delta_1^*$ relative to the MLE risk
$d$ tends to decrease as $\alpha$ increases. This is consistent with
the fact that larger values of $\alpha$ yield smaller shrinkage
functions $r_\alpha(t)$ and therefore weaker interaction between the
two target-centered components.

\begin{figure}[t]
    \centering
    \includegraphics[width=0.78\linewidth]
    {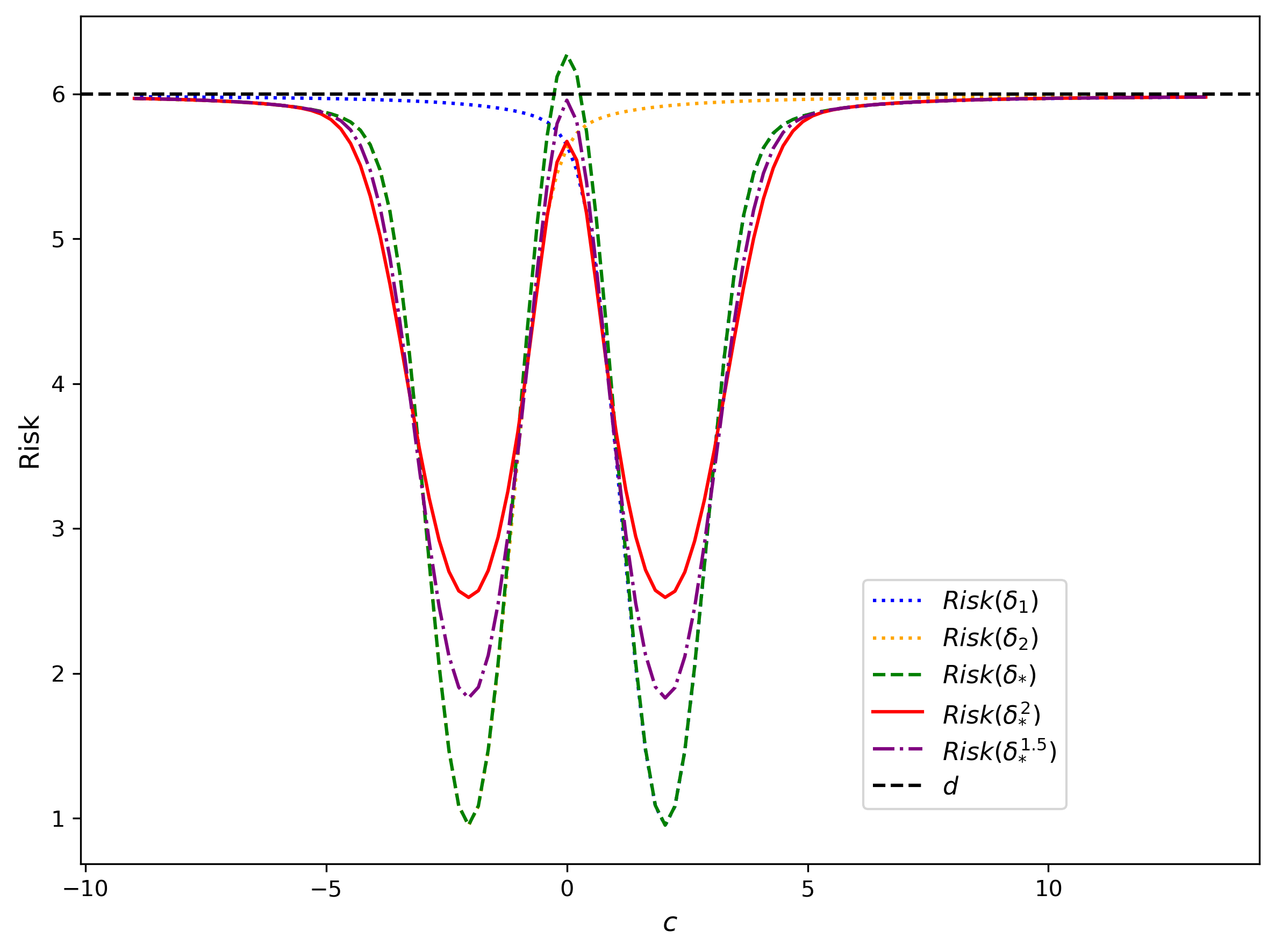}
    \caption{
    Illustration of the effect of rescaling on the estimated risk
    profile of the two-target multiple-shrinkage estimator for
    $d=6$, $\alpha=0.1$, and $D=100$.
    The signed horizontal coordinate is defined by
    $\theta=c\mathbf{1}_d$, where
    $c=\theta^\top\mathbf{1}_d/d$.
    The horizontal dashed line denotes the constant MLE risk $d=6$.
    }
    \label{fig:scaling-illustration}
\end{figure}

\subsection{Effect of the mixing parameter \texorpdfstring{$\alpha$}{alpha}}
\label{subsec:choice-alpha}

For the Strawderman mixing density $h_\alpha(\lambda)
=
(1-\alpha)\lambda^{-\alpha},
\; 0<\lambda<1, \; 0\leq\alpha<1,$ we have $\ell(\lambda)
=
-\frac{\lambda h_\alpha'(\lambda)}{h_\alpha(\lambda)}
=
\alpha.$ To emphasize the dependence on $\alpha$, we write the corresponding shrinkage function as $r_\alpha(t)$. Also
\[
r_\alpha(t)
=
d+2-2\alpha
-
\frac{
2e^{-t/2}
}{
\displaystyle
\int_0^1
\lambda^{d/2-\alpha}
e^{-t\lambda/2}\,d\lambda
}.
\]

For fixed $t>0$, define the tilted density $g_{t,\alpha}(\lambda)
\propto
\lambda^{d/2-\alpha}e^{-t\lambda/2},
\;
0<\lambda<1.$ Then $\frac{r_\alpha(t)}{t}
=
E_{t,\alpha}[\lambda].$ If $\alpha_2>\alpha_1$, then $\frac{g_{t,\alpha_2}(\lambda)}{g_{t,\alpha_1}(\lambda)}
\propto
\lambda^{-(\alpha_2-\alpha_1)},$ which is decreasing in $\lambda$. Hence increasing $\alpha$ shifts the
tilted density toward smaller values of $\lambda$ in the monotone likelihood ratio order, and therefore $r_{\alpha_2}(t)\le r_{\alpha_1}(t)
\;
\text{for all }t\ge0.$ Consequently, if $\alpha_1<\alpha_2$, then $t_{\alpha_1}\le t_{\alpha_2},
\;
t_\alpha
:=
\inf\{t>0:r_\alpha(t)\ge d-2\}.$ This monotonicity is illustrated in
Figure~\ref{fig:rfunctionasvaryingalpha}.

Thus, larger values of $\alpha$ tend to reduce the amount of scaling required by the sufficient condition,
since the threshold $t_\alpha$ increases with $\alpha$. However, this comes with a tradeoff. As $\alpha$ increases, the mixing density places
more mass near $\lambda=0$ (see
Figure~\ref{fig:mixingdensityfunctionasvaryingalpha}), and the induced
marginal density becomes flatter. In this sense, large values of $\alpha$ produce progressively more diffuse marginals, approaching
the behavior of an increasingly flat prior.

\subsection{Choice of \texorpdfstring{$\rho$}{rho}}
\label{subsec:choice-rho}

For the Strawderman family, the admissible range for $\rho$ in Example \ref{ex:rescaled-strawderman} gives
\[
0<\rho<1,
\qquad
\rho
\leq
\rho_{\max}(d,\alpha)
:=
\frac{d-6+2\alpha}{4(2-\alpha)}.
\]
Thus, the feasible range of $\rho$ depends on both the dimension $d$ and the mixing parameter $\alpha$. This restriction arises from the minimaxity argument, rather than from propriety of the prior.

The sufficient scaling bound contains the factor $\left(1+\frac1\rho\right)^2,$ which is strictly decreasing in $\rho>0$. Hence, the smallest
sufficient scaling bound furnished by the theorem is obtained by taking $\rho$ as large as its admissible range permits. If $\rho_{\max}(d,\alpha)<1,$ we may take $\rho=\rho_{\max}(d,\alpha).$ If $\rho_{\max}(d,\alpha)\geq1,$ $\rho$ may be chosen
arbitrarily close to one. In the latter case, $\inf_{0<\rho<1}
\left(1+\frac1\rho\right)^2
=
4.$ 

\subsection{Sufficient scaling range}
\label{subsec:sufficient-scaling}

Let $t_\alpha
=
\sup\{t>0:r_\alpha(t)\leq d-2\}.$ Because $r_\alpha$ is nondecreasing and $\lim_{t\to\infty}r_\alpha(t)
= d+2-2\alpha>d-2$ under the conditions of Example~\ref{ex:rescaled-strawderman},
$t_\alpha$ is finite. The sufficient condition
\[
a
\geq
a_{\mathrm{suff}}(d,D,\alpha,\rho)
:=
\max\left\{
1,\,
\frac{D}{t_\alpha}
\left(1+\frac1\rho\right)^2
\right\},
\]
where
\[
0<\rho<1,
\qquad
\rho\leq
\frac{d-6+2\alpha}{4(2-\alpha)}.
\]

If $\rho_{\max}(d,\alpha)<1$, choosing
$\rho=\rho_{\max}(d,\alpha)$ gives
\begin{align}
a
&\geq
\max\left\{
1,\,
\frac{D}{t_\alpha}
\left(
1+
\frac{4(2-\alpha)}
     {d-6+2\alpha}
\right)^2
\right\}
=
\max\left\{
1,\,
\frac{D}{t_\alpha}
\left(
\frac{d+2-2\alpha}
     {d-6+2\alpha}
\right)^2
\right\}.
\label{eq:strawderman-sufficient-scale}
\end{align}
If $\rho_{\max}(d,\alpha)\geq1$, $\rho$ may instead be chosen
arbitrarily close to one, yielding the limiting sufficient bound $a
>
\max\left\{
1,\,
\frac{4D}{t_\alpha}
\right\}.$ 

The threshold $t_\alpha$ is determined by $r_\alpha(t_\alpha)=d-2.$ Using the integral representation of $r_\alpha$, this is equivalent
to
\[
\int_0^1
\lambda^{d/2-\alpha}
\exp\left\{
\frac{t_\alpha(1-\lambda)}{2}
\right\}
\,d\lambda
=
\frac{1}{2-\alpha}.
\]

We refer to
\eqref{eq:strawderman-sufficient-scale} as a sufficient scaling bound. It guarantees minimaxity globally
over $\mathbb{R}^d$, but it need not be necessary and can therefore be conservative.

\begin{figure}[t]
    \centering

    \begin{subfigure}[h]{0.60\textwidth}
        \centering
        \includegraphics[width=\linewidth]
        {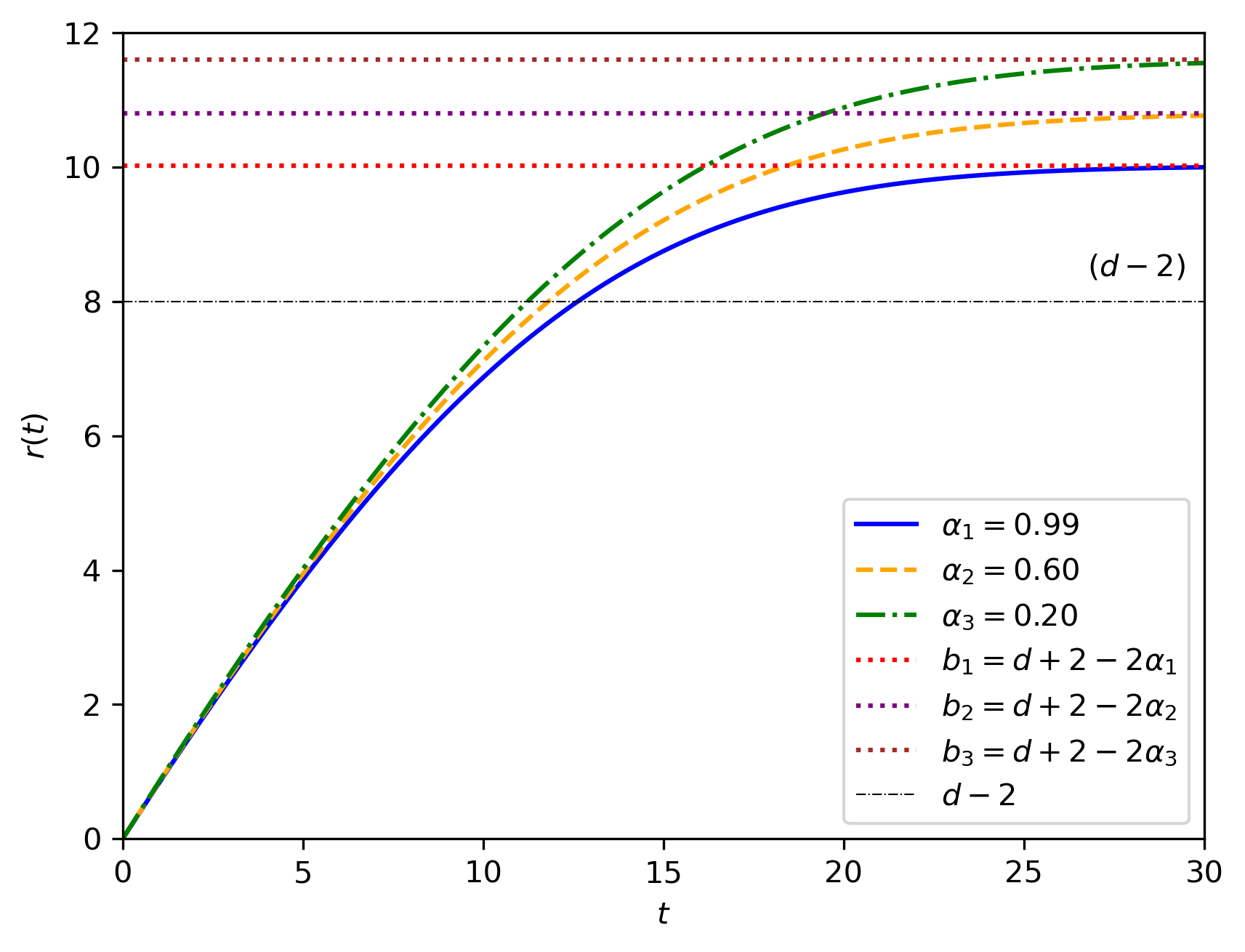}
        \caption{
        Shrinkage functions $r_\alpha(t)$ for
        $\alpha=0.20,0.60,0.99$ when $d=10$.
        }
        \label{fig:rfunctionasvaryingalpha}
    \end{subfigure}

    \vspace{0.6em}

    \begin{subfigure}[h]{0.48\textwidth}
        \centering
        \includegraphics[width=\linewidth]
        {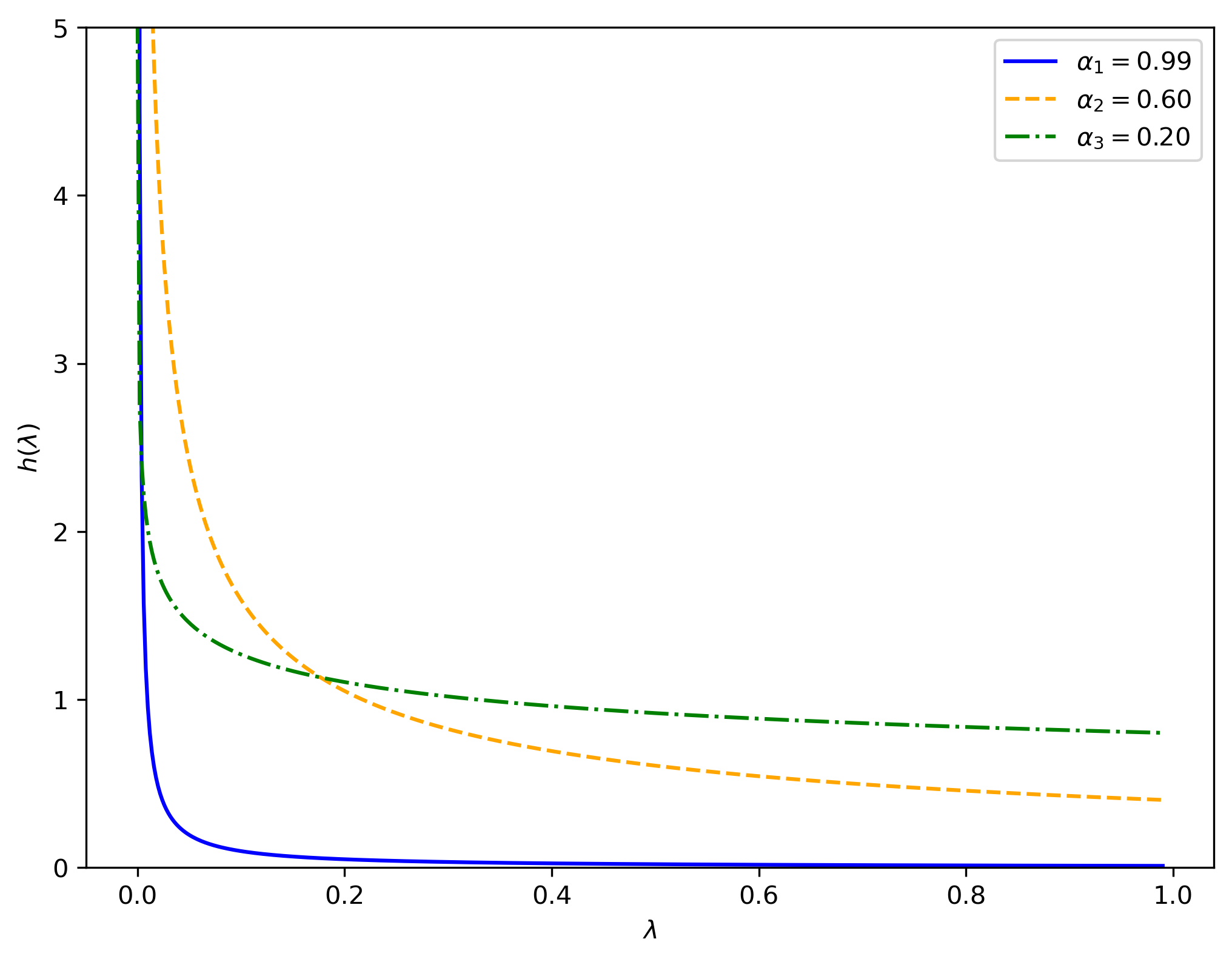}
        \caption{
        Strawderman mixing densities for
        $\alpha=0.20,0.60,0.99$.
        }
        \label{fig:mixingdensityfunctionasvaryingalpha}
    \end{subfigure}
    \hfill
    \begin{subfigure}[h]{0.48\textwidth}
        \centering
        \includegraphics[width=\linewidth]
        {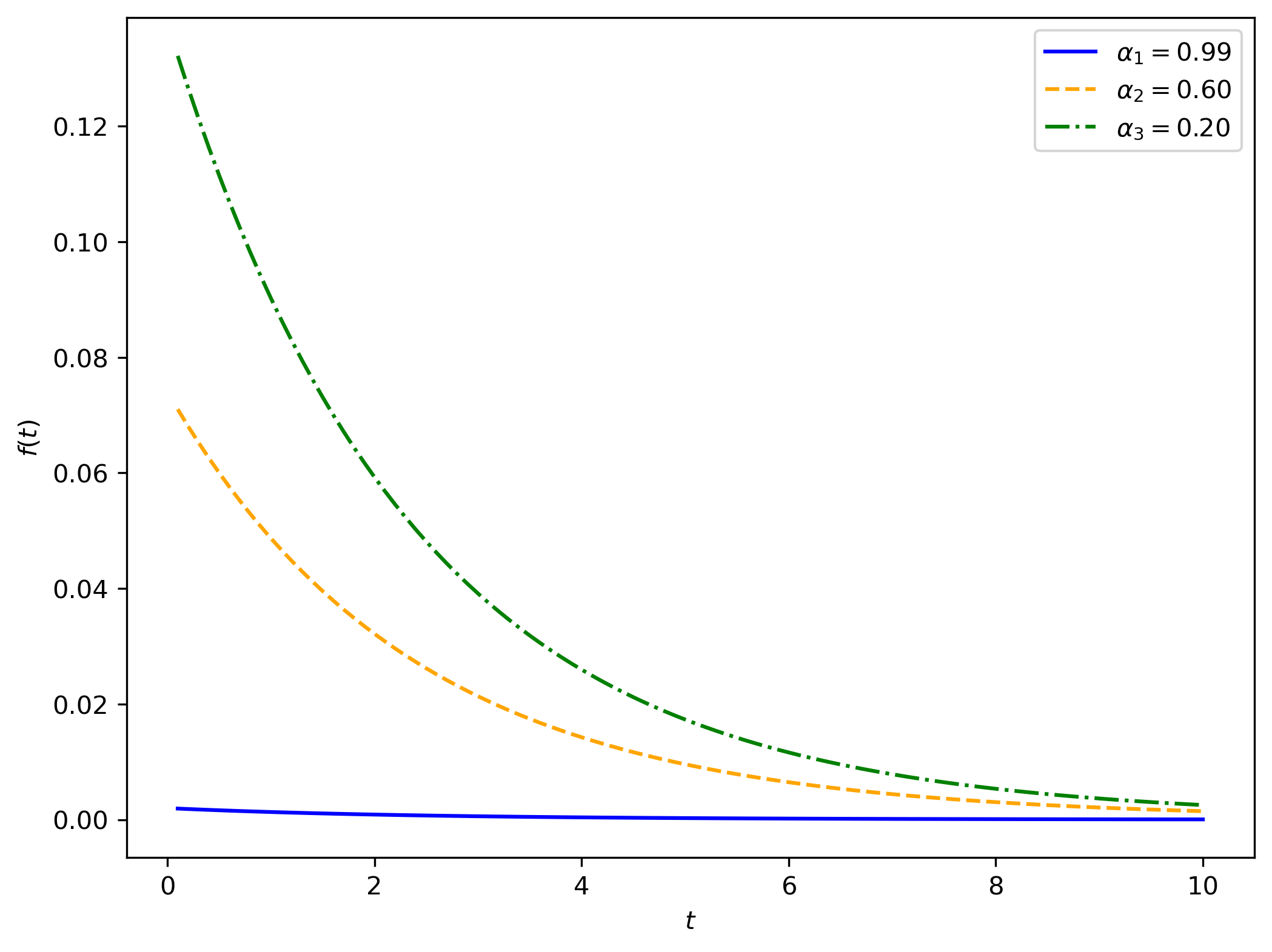}
        \caption{
        Marginal profiles corresponding to
        $\alpha=0.20,0.60,0.99$ when $d=10$.
        }
        \label{fig:margdensityfunctionasvaryingalpha}
    \end{subfigure}

    \caption{
    Effect of the Strawderman mixing parameter $\alpha$.
    Panel (a) shows the corresponding shrinkage functions;
    panel (b) shows the mixing densities; and panel (c) shows
    the induced marginal profiles.
    }
    \label{fig:alpha-comparison}
\end{figure}

\clearpage

\newpage

\section{Concluding remarks}
In conclusion, this article addresses the canonical problem of estimating a multivariate normal mean under squared error loss by focusing on the construction of proper Bayes minimax multiple shrinkage estimators. The use of multiple shrinkage estimators is particularly valuable in practical scenarios where conflicting prior knowledge suggests the potential effectiveness of more than one estimator. By employing specific spherical priors that lead to rescalable marginal densities, the article demonstrates that Stein's minimaxity condition of superharmonicity of the square root of the marginal density provides a feasible approach for constructing Bayes minimax multiple shrinkage estimators. Furthermore, the developed framework enables the construction of proper priors, resulting in admissible minimax multiple shrinkage estimators. Notably, the article establishes that rescaled Strawderman priors yield proper Bayes minimax multiple shrinkage estimators. These findings contribute to the understanding and application of Bayesian methods in the estimation of multivariate normal means, offering practitioners a robust and effective approach to decision-making in the face of conflicting prior knowledge.

An important feature of our proposed proper Bayes minimax multiple shrinkage estimators is that the weights of the underlying proper mixture priors can be meaningfully interpreted as the prior probabilities that the components are the actual generating mechanisms of the targets. This overcomes a drawback of the minimax multiple shrinkage estimators proposed by \cite{george1986a} which, because they were implicitly based on mixtures of improper priors, rendered no such interpretation of the mixture weights, thereby limiting their applicability in practice.

In addition to the class of scale mixtures of multivariate normal priors presented in this article, exploring priors that do not belong to this family can offer further insights and potentially lead to alternative minimax multiple shrinkage estimators. Recently, \cite{fourdrinier2023priors} considered priors which are not variance mixtures of normal distributions. 

Furthermore, the estimation problem discussed in this article is closely linked to the predictive estimation problem for the multivariate normal model, as previously established by \cite{glx2006}. The parallels between the two problems provide a valuable connection, allowing the results obtained in this study to be extended straightforwardly to the predictive estimation problem as well. This extension opens up new avenues for applying the findings of this research in practical settings where prediction is a primary objective.

By considering both alternative priors and extending the results to the predictive estimation problem, future research can continue to advance the field of minimax multiple shrinkage estimation and its applications in multivariate normal models, ultimately enhancing decision-making processes and providing robust solutions in various domains.

\section*{Acknowledgment}

Sadly, Bill Strawderman passed away during the course of this research on October 1, 2024.  The mere existence of minimax multiple shrinkage estimators based on mixtures of Strawderman priors was an open problem that Bill and one of us (Ed) began to wrestle with more than 35 years ago.  Fortunately, Bill never lost interest and ultimately came up with the primary focus of the ideas presented here.  For his intellectual generosity and joy, he will always have our deepest appreciation.

\bibliographystyle{abbrv}
\bibliography{biblio.bib}

\begin{appendix}
\section*{Appendix: Detailed Simulation Results}
\label{appn}

This appendix reports the Monte Carlo risk estimates underlying the
simulation study in Section~\ref{Simulation Study}.  For every configuration,
the targets are placed symmetrically on the $\mathbf 1_d$ axis as
\[
\theta_1=\frac{1}{2}\sqrt{\frac{D}{d}}\,\mathbf 1_d,
\qquad
\theta_2=-\frac{1}{2}\sqrt{\frac{D}{d}}\,\mathbf 1_d,
\qquad
D=\lVert\theta_1-\theta_2\rVert^2.
\]
The risk is evaluated at $\theta=c\mathbf 1_d$, where
$c=\theta^\top\mathbf 1_d/d$.  To use a common column grid for all values of
$d$ and $D$, the tables report the standardized signed coordinate
\[
u:=\frac{c}{\sqrt{D/d}}
\in\{1,0.75,0.50,0.25,0,-0.25,-0.50,-0.75,-1\}.
\]
Equivalently, $c=u\sqrt{D/d}$; thus the two targets occur at
$u=0.50$ and $u=-0.50$.  Each entry is an estimated quadratic risk
$\widehat R(c\mathbf 1_d,\delta)$ based on $10{,}000$ independent draws from
$N_d(c\mathbf 1_d,I_d)$.  The Monte Carlo standard error of every reported
estimate is less than $0.045$.  The notation:
$\delta_1$ and $\delta_2$ are the unrescaled single-target estimators,
$\delta_1^*$ is the unrescaled two-target estimator, and $\delta_a^*$ is its
rescaled version with scaling parameter $a$.

\begin{table}[h]
\centering
\caption{Estimated risks $\widehat R(c\mathbf 1_d,\delta)$ for $d=10$ and $\alpha=0.1$.}
\label{tab:risks_d10_alpha0.1}
\small
\setlength{\tabcolsep}{4.2pt}
\renewcommand{\arraystretch}{0.94}
\begin{tabular}{@{}l*{9}{r}@{}}
\toprule
$u=c/\sqrt{D/d}$ & $1$ & $0.75$ & $0.50$ & $0.25$ & $0$ & $-0.25$ & $-0.50$ & $-0.75$ & $-1$\\
\midrule
\multicolumn{10}{@{}l}{$D=500$}\\
$\delta_1$ & 9.6 &        8.6 &        1.0 &        8.6 &        9.6 &        9.8 &        9.9 &        9.9 &       10.0  \\
$\delta_2$ & 10.0 &       10.0 &        9.9 &        9.8 &        9.6 &        8.6 &        1.0 &        8.6 &        9.6\\
$\delta_1^*$ & 9.6 &        8.6 &        1.0 &        8.6 &       10.2 &        8.6 &        1.0 &        8.6 &        9.6\\
$\delta_{3}^*$ &  9.6 &        7.8 &        5.3 &        7.8 &       10.2 &        7.8 &        5.3 &        7.8 &        9.6  \\
$\delta_{5}^*$ &  9.6 &        8.0 &        7.0 &        8.0 &       10.1 &        8.0 &        7.0 &        8.0 &        9.6 \\
$\delta_{10}^*$ & 9.4 &        8.7 &        8.4 &        8.7 &        9.7 &        8.7 &        8.4 &        8.7 &        9.4  \\
$\delta_{12}^*$ & 9.4 &        8.8 &        8.6 &        8.8 &        9.6 &        8.8 &        8.6 &        8.8 &        9.4    \\
\midrule
\multicolumn{10}{@{}l}{$D=1000$}\\
$\delta_1$ & 9.8 &    9.3 &    1.0 &    9.3 &    9.8 &    9.9 &   10.0 &   10.0 &   10.0  \\
$\delta_2$ &  10.0 &   10.0 &   10.0 &    9.9 &    9.8 &    9.3 &    1.0 &    9.3 &    9.8 \\
$\delta_1^*$ & 9.8 &    9.3 &    1.0 &    9.3 &   10.2 &    9.3 &    1.0 &    9.3 &    9.8\\
$\delta_{3}^*$ &  9.8 &    9.1 &    5.3 &    9.1 &   10.2 &    9.1 &    5.3 &    9.1 &    9.8 \\
$\delta_{5}^*$ &   9.8 &    8.8 &    7.0 &    8.8 &   10.2 &    8.8 &    7.0 &    8.8 &    9.8  \\
$\delta_{10}^*$ & 9.8 &    8.9 &    8.4 &    8.9 &   10.1 &    8.9 &    8.4 &    8.9 &    9.8 \\
$\delta_{12}^*$ & 9.8 &    9.0 &    8.6 &    9.0 &   10.1 &    9.0 &    8.6 &    9.0 &    9.8  \\
\bottomrule
\end{tabular}
\end{table}

\begin{table}[p]
\centering
\caption{Estimated risks $\widehat R(c\mathbf 1_d,\delta)$ for $d=10$ and $\alpha=0.3$.}
\label{tab:risks_d10_alpha0.3}
\small
\setlength{\tabcolsep}{4.2pt}
\renewcommand{\arraystretch}{0.94}
\begin{tabular}{@{}l*{9}{r}@{}}
\toprule
$u=c/\sqrt{D/d}$ & $1$ & $0.75$ & $0.50$ & $0.25$ & $0$ & $-0.25$ & $-0.50$ & $-0.75$ & $-1$\\
\midrule
\multicolumn{10}{@{}l}{$D=500$}\\
$\delta_1$ & 9.6 &        8.6 &        1.0 &        8.5 &        9.6 &        9.8 &        9.9 &        9.9 &       10.0\\
$\delta_2$ & 10.0 &        9.9 &        9.9 &        9.8 &        9.6 &        8.6 &        1.0 &        8.5 &        9.6 \\
$\delta_1^*$ & 9.6 &        8.6 &        1.0 &        8.5 &       10.1 &        8.6 &        1.0 &        8.5 &        9.6 \\
$\delta_{3}^*$ &  9.6 &        7.9 &        5.3 &        7.8 &       10.1 &        7.9 &        5.3 &        7.8 &        9.6   \\
$\delta_{5}^*$ &   9.6 &        8.0 &        7.0 &        8.0 &       10.0 &        8.0 &        7.0 &        8.0 &        9.6 \\
$\delta_{10}^*$ & 9.4 &        8.7 &        8.4 &        8.7 &        9.7 &        8.7 &        8.4 &        8.7 &        9.4  \\
$\delta_{12}^*$ & 9.4 &        8.9 &        8.7 &        8.9 &        9.6 &        8.9 &        8.7 &        8.8 &        9.4 \\
\midrule
\multicolumn{10}{@{}l}{$D=1000$}\\
$\delta_1$ & 9.8 &    9.3 &    1.0 &    9.2 &    9.8 &    9.9 &   10.0 &   10.0 &   10.0\\
$\delta_2$ & 10.0 &   10.0 &   10.0 &    9.9 &    9.8 &    9.3 &    1.0 &    9.2 &    9.8 \\
$\delta_1^*$ & 9.8 &    9.3 &    1.0 &    9.2 &   10.2 &    9.3 &    1.0 &    9.2 &    9.8 \\
$\delta_{3}^*$ &  9.8 &    9.1 &    5.3 &    9.1 &   10.2 &    9.1 &    5.3 &    9.1 &    9.8   \\
$\delta_{5}^*$ &   9.8 &    8.8 &    7.0 &    8.8 &   10.2 &    8.8 &    7.0 &    8.8 &    9.8  \\
$\delta_{10}^*$ & 9.8 &    8.9 &    8.4 &    8.9 &   10.1 &    8.9 &    8.4 &    8.9 &    9.8  \\
$\delta_{12}^*$ & 9.8 &    9.0 &    8.7 &    9.0 &   10.1 &    9.0 &    8.7 &    9.0 &    9.7   \\
\bottomrule
\end{tabular}
\end{table}

\begin{table}[p]
\centering
\caption{Estimated risks $\widehat R(c\mathbf 1_d,\delta)$ for $d=10$ and $\alpha=0.5$.}
\label{tab:risks_d10_alpha0.5}
\small
\setlength{\tabcolsep}{4.2pt}
\renewcommand{\arraystretch}{0.94}
\begin{tabular}{@{}l*{9}{r}@{}}
\toprule
$u=c/\sqrt{D/d}$ & $1$ & $0.75$ & $0.50$ & $0.25$ & $0$ & $-0.25$ & $-0.50$ & $-0.75$ & $-1$\\
\midrule
\multicolumn{10}{@{}l}{$D=500$}\\
$\delta_1$ & 9.6 &        8.5 &        1.1 &        8.5 &        9.6 &        9.8 &        9.9 &        9.9 &       10.0 \\
$\delta_2$ & 10.0 &        9.9 &        9.9 &        9.8 &        9.6 &        8.5 &        1.1 &        8.5 &        9.6 \\
$\delta_1^*$ & 9.6 &        8.5 &        1.1 &        8.5 &       10.1 &        8.5 &        1.1 &        8.5 &        9.6 \\
$\delta_{3}^*$ &  9.6 &        7.9 &        5.4 &        7.9 &       10.1 &        7.9 &        5.4 &        7.9 &        9.6   \\
$\delta_{5}^*$ &  9.6 &        8.0 &        7.0 &        8.0 &       10.0 &        8.0 &        7.0 &        8.0 &        9.5 \\
$\delta_{10}^*$ & 9.4 &        8.7 &        8.4 &        8.7 &        9.7 &        8.7 &        8.4 &        8.7 &        9.4  \\
$\delta_{12}^*$ & 9.4 &        8.9 &        8.7 &        8.9 &        9.6 &        8.9 &        8.7 &        8.9 &        9.4  \\
\midrule
\multicolumn{10}{@{}l}{$D=1000$}\\
$\delta_1$ & 9.8 &    9.2 &    1.1 &    9.2 &    9.8 &    9.9 &   10.0 &   10.0 &   10.0 \\
$\delta_2$ & 10.0 &   10.0 &   10.0 &    9.9 &    9.8 &    9.2 &    1.1 &    9.2 &    9.8 \\
$\delta_1^*$ &  9.8 &    9.2 &    1.1 &    9.2 &   10.1 &    9.2 &    1.1 &    9.2 &    9.8 \\
$\delta_{3}^*$ &  9.8 &    9.1 &    5.4 &    9.0 &   10.1 &    9.1 &    5.4 &    9.0 &    9.8  \\
$\delta_{5}^*$ &  9.8 &    8.8 &    7.0 &    8.8 &   10.1 &    8.8 &    7.0 &    8.8 &    9.8 \\
$\delta_{10}^*$ &  9.8 &    9.0 &    8.4 &    8.9 &   10.1 &    9.0 &    8.4 &    8.9 &    9.8  \\
$\delta_{12}^*$ &  9.8 &    9.0 &    8.7 &    9.0 &   10.0 &    9.0 &    8.7 &    9.0 &    9.7   \\
\bottomrule
\end{tabular}
\end{table}

\begin{table}[p]
\centering
\caption{Estimated risks $\widehat R(c\mathbf 1_d,\delta)$ for $d=10$ and $\alpha=0.7$.}
\label{tab:risks_d10_alpha0.7}
\small
\setlength{\tabcolsep}{4.2pt}
\renewcommand{\arraystretch}{0.94}
\begin{tabular}{@{}l*{9}{r}@{}}
\toprule
$u=c/\sqrt{D/d}$ & $1$ & $0.75$ & $0.50$ & $0.25$ & $0$ & $-0.25$ & $-0.50$ & $-0.75$ & $-1$\\
\midrule
\multicolumn{10}{@{}l}{$D=500$}\\
$\delta_1$ &  9.6 &        8.4 &        1.2 &        8.4 &        9.6 &        9.8 &        9.9 &        9.9 &       10.0 \\
$\delta_2$ & 10.0 &        9.9 &        9.9 &        9.8 &        9.6 &        8.4 &        1.2 &        8.4 &        9.6\\
$\delta_1^*$ &  9.6 &        8.4 &        1.2 &        8.4 &       10.0 &        8.4 &        1.2 &        8.4 &        9.6\\
$\delta_{3}^*$ &  9.6 &        7.9 &        5.4 &        7.9 &       10.0 &        7.9 &        5.4 &        7.9 &        9.6  \\
$\delta_{5}^*$ &  9.5 &        8.1 &        7.0 &        8.0 &       10.0 &        8.1 &        7.0 &        8.0 &        9.5 \\
$\delta_{10}^*$ &  9.4 &        8.7 &        8.4 &        8.7 &        9.7 &        8.7 &        8.4 &        8.7 &        9.4  \\
$\delta_{12}^*$ &  9.4 &        8.9 &        8.7 &        8.9 &        9.6 &        8.9 &        8.7 &        8.9 &        9.4 \\
\midrule
\multicolumn{10}{@{}l}{$D=1000$}\\
$\delta_1$ &  9.8 &    9.2 &    1.2 &    9.2 &    9.8 &    9.9 &   10.0 &   10.0 &   10.0 \\
$\delta_2$ & 10.0 &   10.0 &   10.0 &    9.9 &    9.8 &    9.2 &    1.2 &    9.2 &    9.8\\
$\delta_1^*$ &  9.8 &    9.2 &    1.2 &    9.2 &   10.1 &    9.2 &    1.2 &    9.2 &    9.8\\
$\delta_{3}^*$ &  9.8 &    9.0 &    5.4 &    9.0 &   10.1 &    9.0 &    5.4 &    9.0 &    9.8  \\
$\delta_{5}^*$ &  9.8 &    8.8 &    7.0 &    8.8 &   10.1 &    8.8 &    7.0 &    8.8 &    9.8 \\
$\delta_{10}^*$ &  9.8 &    9.0 &    8.4 &    9.0 &   10.1 &    9.0 &    8.4 &    9.0 &    9.8  \\
$\delta_{12}^*$ &  9.8 &    9.1 &    8.7 &    9.1 &   10.0 &    9.1 &    8.7 &    9.1 &    9.7   \\
\bottomrule
\end{tabular}
\end{table}

\begin{table}[p]
\centering
\caption{Estimated risks $\widehat R(c\mathbf 1_d,\delta)$ for $d=10$ and $\alpha=0.9$.}
\label{tab:risks_d10_alpha0.9}
\small
\setlength{\tabcolsep}{4.2pt}
\renewcommand{\arraystretch}{0.94}
\begin{tabular}{@{}l*{9}{r}@{}}
\toprule
$u=c/\sqrt{D/d}$ & $1$ & $0.75$ & $0.50$ & $0.25$ & $0$ & $-0.25$ & $-0.50$ & $-0.75$ & $-1$\\
\midrule
\multicolumn{10}{@{}l}{$D=500$}\\
$\delta_1$ &  9.6 &        8.4 &        1.3 &        8.4 &        9.6 &        9.8 &        9.9 &        9.9 &       10.0  \\
$\delta_2$ & 10.0 &        9.9 &        9.9 &        9.8 &        9.6 &        8.4 &        1.3 &        8.4 &        9.6 \\
$\delta_1^*$ & 9.6 &        8.4 &        1.3 &        8.4 &       10.0 &        8.4 &        1.3 &        8.4 &        9.6\\
$\delta_{3}^*$ &  9.6 &        7.9 &        5.4 &        7.9 &       10.0 &        7.9 &        5.4 &        7.9 &        9.6  \\
$\delta_{5}^*$ &  9.5 &        8.1 &        7.0 &        8.1 &        9.9 &        8.1 &        7.0 &        8.1 &        9.5  \\
$\delta_{10}^*$ &  9.4 &        8.7 &        8.4 &        8.7 &        9.7 &        8.7 &        8.4 &        8.7 &        9.4  \\
$\delta_{12}^*$ &  9.4 &        8.9 &        8.7 &        8.9 &        9.6 &        8.9 &        8.7 &        8.9 &        9.4   \\
\midrule
\multicolumn{10}{@{}l}{$D=1000$}\\
$\delta_1$ &  9.8 &    9.2 &    1.3 &    9.1 &    9.8 &    9.9 &   10.0 &   10.0 &   10.0  \\
$\delta_2$ & 10.0 &   10.0 &   10.0 &    9.9 &    9.8 &    9.2 &    1.3 &    9.1 &    9.8  \\
$\delta_1^*$ & 9.8 &    9.2 &    1.3 &    9.1 &   10.1 &    9.2 &    1.3 &    9.1 &    9.8\\
$\delta_{3}^*$ &  9.8 &    9.0 &    5.4 &    9.0 &   10.1 &    9.0 &    5.4 &    9.0 &    9.8  \\
$\delta_{5}^*$ &   9.8 &    8.9 &    7.0 &    8.8 &   10.1 &    8.9 &    7.0 &    8.8 &    9.8 \\
$\delta_{10}^*$ &  9.8 &    9.0 &    8.4 &    9.0 &   10.0 &    9.0 &    8.4 &    9.0 &    9.8   \\
$\delta_{12}^*$ &  9.8 &    9.1 &    8.7 &    9.1 &   10.0 &    9.1 &    8.7 &    9.1 &    9.7   \\
\bottomrule
\end{tabular}
\end{table}

\begin{table}[p]
\centering
\caption{Estimated risks $\widehat R(c\mathbf 1_d,\delta)$ for $d=6$ and $\alpha=0.1$.}
\label{tab:risks_d6_alpha0.1}
\small
\setlength{\tabcolsep}{4.2pt}
\renewcommand{\arraystretch}{0.94}
\begin{tabular}{@{}l*{9}{r}@{}}
\toprule
$u=c/\sqrt{D/d}$ & $1$ & $0.75$ & $0.50$ & $0.25$ & $0$ & $-0.25$ & $-0.50$ & $-0.75$ & $-1$\\
\midrule
\multicolumn{10}{@{}l}{$D=100$}\\
$\delta_1$ &  5.8 &        3.5 &        0.8 &        3.5 &        5.8 &        6.0 &        6.0 &        6.0 &        6.0 \\
$\delta_2$ & 6.0 &        6.0 &        6.0 &        6.0 &        5.8 &        3.5 &        0.8 &        3.5 &        5.8\\
$\delta_1^*$ & 5.8 &        3.5 &        0.8 &        3.6 &        6.5 &        3.6 &        0.8 &        3.5 &        5.8\\
$\delta_{3}^*$ &  5.1 &        3.9 &        3.4 &        4.0 &        5.4 &        4.0 &        3.4 &        3.9 &        5.0   \\
$\delta_{5}^*$ &  5.0 &        4.5 &        4.3 &        4.6 &        5.2 &        4.6 &        4.3 &        4.5 &        5.0 \\
$\delta_{10}^*$ & 5.3 &        5.1 &        5.1 &        5.2 &        5.4 &        5.2 &        5.1 &        5.1 &        5.3 \\
\midrule
\multicolumn{10}{@{}l}{$D=200$}\\
$\delta_1$ &  6.0 &        4.9 &        0.8 &        4.9 &        6.0 &        6.0 &        6.0 &        6.0 &        6.0 \\
$\delta_2$ &6.0 &        6.0 &        6.0 &        6.0 &        6.0 &        4.9 &        0.8 &        4.9 &        6.0\\
$\delta_1^*$ & 6.0 &        4.9 &        0.8 &        4.9 &        6.6 &        4.9 &        0.8 &        4.9 &        6.0 \\
$\delta_{3}^*$ &  5.8 &        4.3 &        3.4 &        4.4 &        6.2 &        4.4 &        3.4 &        4.3 &        5.8   \\
$\delta_{5}^*$ &   5.5 &        4.7 &        4.3 &        4.7 &        5.8 &        4.7 &        4.3 &        4.7 &        5.5 \\
$\delta_{10}^*$ & 5.5 &        5.2 &        5.1 &        5.2 &        5.6 &        5.2 &        5.1 &        5.2 &        5.5 \\
$\delta_{12}^*$ &  5.5 &        5.3 &        5.2 &        5.4 &        5.6 &        5.4 &        5.2 &        5.3 &        5.5 \\
\midrule
\multicolumn{10}{@{}l}{$D=500$}\\
$\delta_1$ &  6.0 &        5.9 &        0.8 &        5.9 &        6.0 &        6.0 &        6.0 &        6.0 &        6.0   \\
$\delta_2$ &  6.0 &        6.0 &        6.0 &        6.0 &        6.0 &        5.9 &        0.8 &        5.9 &        6.0 \\
$\delta_1^*$ & 6.0 &        5.9 &        0.8 &        5.9 &        6.3 &        5.9 &        0.8 &        5.9 &        6.0 \\
$\delta_{3}^*$ &  6.0 &        5.3 &        3.4 &        5.3 &        6.3 &        5.3 &        3.4 &        5.3 &        6.0  \\
$\delta_{5}^*$ &   6.0 &        5.2 &        4.3 &        5.2 &        6.3 &        5.2 &        4.3 &        5.2 &        6.0  \\
$\delta_{10}^*$ &  5.8 &        5.3 &        5.1 &        5.3 &        6.1 &        5.3 &        5.1 &        5.3 &        5.8 \\
$\delta_{12}^*$ &  5.8 &        5.4 &        5.2 &        5.4 &        6.0 &        5.4 &        5.2 &        5.4 &        5.8 \\
\bottomrule
\end{tabular}
\end{table}

\begin{table}[p]
\centering
\caption{Estimated risks $\widehat R(c\mathbf 1_d,\delta)$ for $d=6$ and $\alpha=0.3$.}
\label{tab:risks_d6_alpha0.3}
\small
\setlength{\tabcolsep}{4.2pt}
\renewcommand{\arraystretch}{0.94}
\begin{tabular}{@{}l*{9}{r}@{}}
\toprule
$u=c/\sqrt{D/d}$ & $1$ & $0.75$ & $0.50$ & $0.25$ & $0$ & $-0.25$ & $-0.50$ & $-0.75$ & $-1$\\
\midrule
\multicolumn{10}{@{}l}{$D=100$}\\
$\delta_1$ & 5.7 &        3.5 &        0.9 &        3.5 &        5.7 &        5.9 &        5.9 &        6.0 &        6.0\\
$\delta_2$ &  6.0 &        6.0 &        6.0 &        5.9 &        5.7 &        3.5 &        0.9 &        3.5 &        5.7 \\
$\delta_1^*$ & 5.7 &        3.5 &        0.9 &        3.6 &        6.4 &        3.6 &        0.9 &        3.5 &        5.7 \\
$\delta_{3}^*$ &  5.1 &        3.9 &        3.4 &        4.0 &        5.4 &        4.0 &        3.4 &        3.9 &        5.1  \\
$\delta_{5}^*$ &   5.0 &        4.5 &        4.3 &        4.6 &        5.2 &        4.6 &        4.3 &        4.5 &        5.0 \\
$\delta_{10}^*$ & 5.3 &        5.2 &        5.1 &        5.3 &        5.4 &        5.3 &        5.1 &        5.2 &        5.3  \\
\midrule
\multicolumn{10}{@{}l}{$D=200$}\\
$\delta_1$ & 5.9 &        4.8 &        0.9 &        4.8 &        5.9 &        5.9 &        6.0 &        6.0 &        6.0\\
$\delta_2$ &   6.0 &        6.0 &        6.0 &        6.0 &        5.9 &        4.8 &        0.9 &        4.8 &        5.9 \\
$\delta_1^*$ &  5.9 &        4.8 &        0.9 &        4.8 &        6.5 &        4.9 &        0.9 &        4.8 &        5.9 \\
$\delta_{0.5}^*$ & 5.9 &        5.4 &        0.5 &        5.4 &        6.5 &        5.5 &        0.5 &        5.4 &        5.9 \\
$\delta_{3}^*$ &  5.7 &        4.4 &        3.4 &        4.4 &        6.2 &        4.4 &        3.4 &        4.4 &        5.7  \\
$\delta_{5}^*$ &   5.5 &        4.7 &        4.3 &        4.7 &        5.8 &        4.7 &        4.3 &        4.7 &        5.5 \\
$\delta_{10}^*$ & 5.5 &        5.2 &        5.1 &        5.3 &        5.6 &        5.3 &        5.1 &        5.2 &        5.5 \\
$\delta_{12}^*$ & 5.5 &        5.3 &        5.2 &        5.4 &        5.6 &        5.4 &        5.2 &        5.3 &        5.5  \\
\midrule
\multicolumn{10}{@{}l}{$D=500$}\\
$\delta_1$ & 6.0 &        5.8 &        0.9 &        5.8 &        6.0 &        6.0 &        6.0 &        6.0 &        6.0\\
$\delta_2$ & 6.0 &        6.0 &        6.0 &        6.0 &        6.0 &        5.8 &        0.9 &        5.8 &        6.0 \\
$\delta_1^*$ &  6.0 &        5.8 &        0.9 &        5.8 &        6.3 &        5.8 &        0.9 &        5.8 &        6.0 \\
$\delta_{3}^*$ &  6.0 &        5.3 &        3.4 &        5.3 &        6.3 &        5.3 &        3.4 &        5.3 &        6.0   \\
$\delta_{5}^*$ &   5.9 &        5.2 &        4.3 &        5.2 &        6.2 &        5.2 &        4.3 &        5.2 &        5.9   \\
$\delta_{10}^*$ & 5.8 &        5.3 &        5.1 &        5.3 &        6.1 &        5.4 &        5.1 &        5.3 &        5.8  \\
$\delta_{12}^*$ & 5.8 &        5.4 &        5.2 &        5.4 &        6.0 &        5.4 &        5.2 &        5.4 &        5.8  \\
\bottomrule
\end{tabular}
\end{table}

\begin{table}[p]
\centering
\caption{Estimated risks $\widehat R(c\mathbf 1_d,\delta)$ for $d=6$ and $\alpha=0.5$.}
\label{tab:risks_d6_alpha0.5}
\small
\setlength{\tabcolsep}{4.2pt}
\renewcommand{\arraystretch}{0.94}
\begin{tabular}{@{}l*{9}{r}@{}}
\toprule
$u=c/\sqrt{D/d}$ & $1$ & $0.75$ & $0.50$ & $0.25$ & $0$ & $-0.25$ & $-0.50$ & $-0.75$ & $-1$\\
\midrule
\multicolumn{10}{@{}l}{$D=100$}\\
$\delta_1$ & 5.6 &        3.5 &        1.0 &        3.5 &        5.6 &        5.9 &        5.9 &        5.9 &        6.0 \\
$\delta_2$ & 6.0 &        5.9 &        5.9 &        5.9 &        5.6 &        3.5 &        1.0 &        3.5 &        5.6 \\
$\delta_1^*$ & 5.6 &        3.5 &        1.0 &        3.6 &        6.3 &        3.6 &        1.0 &        3.5 &        5.6 \\
$\delta_{3}^*$ &  5.1 &        3.9 &        3.4 &        4.1 &        5.4 &        4.1 &        3.4 &        3.9 &        5.1   \\
$\delta_{5}^*$ &  5.1 &        4.5 &        4.3 &        4.7 &        5.2 &        4.7 &        4.3 &        4.5 &        5.0  \\
$\delta_{10}^*$ & 5.3 &        5.2 &        5.1 &        5.3 &        5.4 &        5.3 &        5.1 &        5.2 &        5.3 \\
\midrule
\multicolumn{10}{@{}l}{$D=200$}\\
$\delta_1$ & 5.9 &        4.8 &        1.0 &        4.8 &        5.8 &        5.9 &        6.0 &        6.0 &        6.0  \\
$\delta_2$ & 6.0 &        6.0 &        6.0 &        5.9 &        5.9 &        4.8 &        1.0 &        4.8 &        5.8 \\
$\delta_1^*$ &  5.9 &        4.8 &        1.0 &        4.8 &        6.4 &        4.8 &        1.0 &        4.8 &        5.8 \\
$\delta_{3}^*$ &  5.7 &        4.4 &        3.4 &        4.4 &        6.1 &        4.4 &        3.4 &        4.4 &        5.7   \\
$\delta_{5}^*$ &  5.5 &        4.7 &        4.3 &        4.7 &        5.8 &        4.8 &        4.3 &        4.7 &        5.5 \\
$\delta_{10}^*$ &5.5 &        5.2 &        5.1 &        5.3 &        5.6 &        5.3 &        5.1 &        5.2 &        5.5 \\
$\delta_{12}^*$ & 5.5 &        5.3 &        5.3 &        5.4 &        5.6 &        5.4 &        5.3 &        5.3 &        5.5  \\
\midrule
\multicolumn{10}{@{}l}{$D=500$}\\
$\delta_1$ & 5.9 &        5.8 &        1.0 &        5.7 &        5.9 &        6.0 &        6.0 &        6.0 &        6.0 \\
$\delta_2$ & 6.0 &        6.0 &        6.0 &        6.0 &        5.9 &        5.8 &        1.0 &        5.7 &        5.9 \\
$\delta_1^*$ &  5.9 &        5.8 &        1.0 &        5.7 &        6.2 &        5.8 &        1.0 &        5.7 &        5.9  \\
$\delta_{3}^*$ &  5.9 &        5.3 &        3.4 &        5.3 &        6.2 &        5.3 &        3.4 &        5.3 &        5.9  \\
$\delta_{5}^*$ &  5.9 &        5.2 &        4.3 &        5.2 &        6.2 &        5.2 &        4.3 &        5.2 &        5.9 \\
$\delta_{10}^*$ & 5.8 &        5.4 &        5.1 &        5.4 &        6.0 &        5.4 &        5.1 &        5.4 &        5.8  \\
$\delta_{12}^*$ & 5.8 &        5.4 &        5.2 &        5.4 &        6.0 &        5.4 &        5.2 &        5.4 &        5.8\\
\bottomrule
\end{tabular}
\end{table}

\begin{table}[p]
\centering
\caption{Estimated risks $\widehat R(c\mathbf 1_d,\delta)$ for $d=6$ and $\alpha=0.7$.}
\label{tab:risks_d6_alpha0.7}
\small
\setlength{\tabcolsep}{4.2pt}
\renewcommand{\arraystretch}{0.94}
\begin{tabular}{@{}l*{9}{r}@{}}
\toprule
$u=c/\sqrt{D/d}$ & $1$ & $0.75$ & $0.50$ & $0.25$ & $0$ & $-0.25$ & $-0.50$ & $-0.75$ & $-1$\\
\midrule
\multicolumn{10}{@{}l}{$D=100$}\\
$\delta_1$ &  5.6 &        3.5 &        1.0 &        3.5 &        5.6 &        5.8 &        5.9 &        5.9 &        5.9 \\
$\delta_2$ & 6.0 &        5.9 &        5.9 &        5.8 &        5.6 &        3.5 &        1.0 &        3.5 &        5.6\\
$\delta_1^*$ &  5.6 &        3.5 &        1.0 &        3.6 &        6.1 &        3.6 &        1.0 &        3.5 &        5.6\\
$\delta_{3}^*$ &  5.1 &        4.0 &        3.5 &        4.1 &        5.4 &        4.1 &        3.5 &        4.0 &        5.1   \\
$\delta_{5}^*$ &  5.1 &        4.6 &        4.4 &        4.7 &        5.3 &        4.7 &        4.4 &        4.5 &        5.1 \\
$\delta_{10}^*$ &  5.3 &        5.2 &        5.2 &        5.3 &        5.4 &        5.3 &        5.2 &        5.2 &        5.3  \\
\midrule
\multicolumn{10}{@{}l}{$D=200$}\\
$\delta_1$ &   5.8 &        4.8 &        1.0 &        4.7 &        5.8 &        5.9 &        5.9 &        6.0 &        6.0 \\
$\delta_2$ & 6.0 &        6.0 &        5.9 &        5.9 &        5.8 &        4.8 &        1.0 &        4.7 &        5.8\\
$\delta_1^*$ &  5.8 &        4.8 &        1.0 &        4.8 &        6.3 &        4.8 &        1.0 &        4.7 &        5.8\\
$\delta_{3}^*$ &   5.7 &        4.4 &        3.5 &        4.4 &        6.1 &        4.5 &        3.5 &        4.4 &        5.7  \\
$\delta_{5}^*$ &  5.5 &        4.7 &        4.4 &        4.8 &        5.8 &        4.8 &        4.4 &        4.7 &        5.5 \\
$\delta_{10}^*$ &  5.5 &        5.2 &        5.1 &        5.3 &        5.6 &        5.3 &        5.1 &        5.2 &        5.5  \\
$\delta_{12}^*$ &  5.5 &        5.3 &        5.3 &        5.4 &        5.6 &        5.4 &        5.3 &        5.3 &        5.5  \\
\midrule
\multicolumn{10}{@{}l}{$D=500$}\\
$\delta_1$ &  5.9 &        5.7 &        1.0 &        5.7 &        5.9 &        6.0 &        6.0 &        6.0 &        6.0 \\
$\delta_2$ & 6.0 &        6.0 &        6.0 &        6.0 &        5.9 &        5.7 &        1.0 &        5.7 &        5.9 \\
$\delta_1^*$ &  5.9 &        5.7 &        1.0 &        5.7 &        6.2 &        5.7 &        1.0 &        5.7 &        5.9\\
$\delta_{3}^*$ &  5.9 &        5.3 &        3.5 &        5.3 &        6.2 &        5.3 &        3.5 &        5.3 &        5.9  \\
$\delta_{5}^*$ &  5.9 &        5.2 &        4.3 &        5.2 &        6.2 &        5.2 &        4.3 &        5.2 &        5.9 \\
$\delta_{10}^*$ &  5.8 &        5.4 &        5.1 &        5.4 &        6.0 &        5.4 &        5.1 &        5.4 &        5.8  \\
$\delta_{12}^*$ &  5.8 &        5.4 &        5.3 &        5.4 &        6.0 &        5.4 &        5.3 &        5.4 &        5.8 \\
\bottomrule
\end{tabular}
\end{table}

\begin{table}[p]
\centering
\caption{Estimated risks $\widehat R(c\mathbf 1_d,\delta)$ for $d=6$ and $\alpha=0.9$.}
\label{tab:risks_d6_alpha0.9}
\small
\setlength{\tabcolsep}{4.2pt}
\renewcommand{\arraystretch}{0.94}
\begin{tabular}{@{}l*{9}{r}@{}}
\toprule
$u=c/\sqrt{D/d}$ & $1$ & $0.75$ & $0.50$ & $0.25$ & $0$ & $-0.25$ & $-0.50$ & $-0.75$ & $-1$\\
\midrule
\multicolumn{10}{@{}l}{$D=100$}\\
$\delta_1$ &  5.5 &        3.5 &        1.1 &        3.5 &        5.5 &        5.8 &        5.9 &        5.9 &        5.9   \\
$\delta_2$ & 5.9 &        5.9 &        5.9 &        5.8 &        5.5 &        3.5 &        1.1 &        3.5 &        5.5\\
$\delta_1^*$ & 5.5 &        3.5 &        1.1 &        3.6 &        6.0 &        3.7 &        1.1 &        3.5 &        5.5\\
$\delta_{3}^*$ &  5.1 &        4.0 &        3.5 &        4.2 &        5.4 &        4.2 &        3.5 &        4.0 &        5.1  \\
$\delta_{5}^*$ &  5.1 &        4.6 &        4.4 &        4.7 &        5.3 &        4.7 &        4.4 &        4.6 &        5.1  \\
$\delta_{10}^*$ &  5.3 &        5.2 &        5.2 &        5.3 &        5.4 &        5.3 &        5.2 &        5.2 &        5.3 \\
\midrule
\multicolumn{10}{@{}l}{$D=200$}\\
$\delta_1$ &  5.8 &        4.7 &        1.1 &        4.7 &        5.8 &        5.9 &        5.9 &        6.0 &        6.0    \\
$\delta_2$ & 6.0 &        6.0 &        5.9 &        5.9 &        5.8 &        4.7 &        1.1 &        4.7 &        5.8\\
$\delta_1^*$ & 5.8 &        4.7 &        1.1 &        4.7 &        6.2 &        4.8 &        1.1 &        4.7 &        5.8 \\
$\delta_{3}^*$ &  5.7 &        4.5 &        3.5 &        4.5 &        6.0 &        4.5 &        3.5 &        4.5 &        5.7 \\
$\delta_{5}^*$ &  5.5 &        4.8 &        4.4 &        4.8 &        5.8 &        4.8 &        4.4 &        4.8 &        5.5  \\
$\delta_{10}^*$ &  5.5 &        5.2 &        5.1 &        5.3 &        5.6 &        5.3 &        5.1 &        5.2 &        5.5 \\
$\delta_{12}^*$ &  5.5 &        5.3 &        5.3 &        5.4 &        5.6 &        5.4 &        5.3 &        5.3 &        5.5 \\
\midrule
\multicolumn{10}{@{}l}{$D=500$}\\
$\delta_1$ &  5.9 &        5.6 &        1.1 &        5.6 &        5.9 &        5.9 &        6.0 &        6.0 &        6.0  \\
$\delta_2$ & 6.0 &        6.0 &        6.0 &        6.0 &        5.9 &        5.6 &        1.1 &        5.6 &        5.9  \\
$\delta_1^*$ & 5.9 &        5.6 &        1.1 &        5.6 &        6.1 &        5.7 &        1.1 &        5.6 &        5.9 \\
$\delta_{3}^*$ &  5.9 &        5.3 &        3.5 &        5.3 &        6.1 &        5.3 &        3.5 &        5.3 &        5.9  \\
$\delta_{5}^*$ &   5.9 &        5.2 &        4.4 &        5.2 &        6.1 &        5.2 &        4.4 &        5.2 &        5.9 \\
$\delta_{10}^*$ &  5.8 &        5.4 &        5.1 &        5.4 &        6.0 &        5.4 &        5.1 &        5.4 &        5.8  \\
$\delta_{12}^*$ &  5.8 &        5.4 &        5.3 &        5.5 &        6.0 &        5.5 &        5.3 &        5.4 &        5.8 \\
\bottomrule
\end{tabular}
\end{table}

\end{appendix}

\end{document}